\documentclass[reqno]{amsart}
\usepackage{amsmath,amssymb,amsthm}
\usepackage[cmtip,color,all]{xy}
\xyoption{tips}
\SelectTips{cm}{10}
\usepackage[dvipsnames,table,xcdraw]{xcolor}
\usepackage{enumitem}
\usepackage{mathtools}
\usepackage{microtype}
\usepackage{booktabs}
\usepackage{caption}
\usepackage{diagbox}
\usepackage{mathrsfs}
\usepackage{hyperref}
\usepackage{tikz-cd}
\usetikzlibrary{arrows.meta}
\usetikzlibrary{shapes.geometric}
\usepackage{bbold}
\usepackage{floatrow}
\usepackage{pifont}
\usepackage{comment}
\usepackage{geometry}
\usepackage{setspace}
\usepackage{stmaryrd,xcolor,soul,tikz}
\usepackage{pdfpages}

\definecolor{gold}{rgb}{1.0, 0.87, 0.4}

\newcommand{\C}{\mathcal{C}}

\newcommand{\Mor}{\mathrm{Mor}}

\newcommand{\AF}{\mathsf{AF}}
\newcommand{\AC}{\mathsf{AC}}
\newcommand{\W}{\mathsf{W}}

\newcommand{\F}{\mathsf{F}}
\newcommand{\Cof}{\mathsf{C}}
\newcommand{\Cate}{\mathcal{C}}

\usepackage{tikz}
\usetikzlibrary{positioning,arrows}
\usetikzlibrary{shapes.geometric, calc}
\usetikzlibrary{decorations.pathreplacing}

\newtheorem{theorem}{Theorem}[section]

\newtheorem{lemma}[theorem]{Lemma}
\newtheorem{proposition}[theorem]{Proposition}

\newtheorem*{theorem*}{Theorem}

\theoremstyle{definition}
\newtheorem{definition}[theorem]{Definition}
\newtheorem{example}[theorem]{Example}
\newtheorem{remark}[theorem]{Remark}
\newtheorem{convention}[theorem]{Convention}
\newtheorem{construction}[theorem]{Construction}
\usepackage{hyperref}

\definecolor{dark-red}{rgb}{0.5,0.15,0.15}
\definecolor{dark-blue}{rgb}{0.15,0.15,0.6}
\definecolor{dark-green}{rgb}{0.15,0.6,0.15}
\hypersetup{
    colorlinks, linkcolor=dark-red,
    citecolor=dark-blue, urlcolor=dark-green
}

\definecolor{gRed}{HTML}{ff5100}
\definecolor{gGreen}{HTML}{2b83ba}

\usepackage{fancyhdr}           
\fancypagestyle{plain}{
    \lhead{}
    \fancyhead[R]{\thepage}
    \fancyhead[L]{}
    
    \fancyfoot{}
}

\usepackage[textwidth=3cm, textsize=small, colorinlistoftodos]{todonotes}

\DeclareMathOperator{\id}{id} 

\allowdisplaybreaks
\graphicspath{{./images/}}

\title{Left and Right Bousfield Localization on Lattices}
\author{Andrés Carnero Bravo}
\address{Centro de Ciencias Matemáticas, UNAM, Ap. Postal 61-3 Xangari, 58089 Morelia, Michoacán, México}
\email{carnero@matmor.unam.mx}
\author{Shuchita Goyal}
\address{Birla Institute of Technology And Science, Pilani, India}
\email{shuchita.goyal@pilani.bits-pilani.ac.in}
\author{Sofía Martínez Alberga}
\address{Stonehill College,Department of Mathematics, North Easton, MA, USA}
\email{smartinezalbe@stonehill.edu}
\author{\\Cherry Ng}
\address{Northwestern University, Department of Mathematics, Evanston, IL, USA}
\email{cherry.ng@northwestern.edu}
\author{Constanze Roitzheim}
\address{University of Kent, SMSAS, Canterbury CT2 7FS, UK}
\email{c.roitzheim@kent.ac.uk}
\author{Daniel Tolosa}
\address{Arizona State University, Tempe, AZ, USA}
\email{dtolosav@asu.edu}
\date{\today}

\begin{document}

\begin{abstract}
The key information of a model category structure on a poset is encoded in a transfer system, which is a combinatorial gadget, originally introduced to investigate homotopy coherence structures in equivariant homotopy theory. 
We describe how a transfer system associated with in a model structure on a lattice is affected by left and right Bousfield localization and provide a minimal generating system of morphisms which are responsible for the change in model structure. 
This leads to new concrete insights into the behavior of model categories on posets in general. 
\end{abstract}

\maketitle

\section*{Introduction}

Transfer systems originally emerged in equivariant homotopy theory with the goal of understanding equivariant analogs of higher coherences.
Nonequivariantly, May \cite{May72} showed that homotopy-coherent commutative and associative multiplicative structures were in direct and unique correspondence with iterated loop spaces. 
These are now commonly known as $E_{\infty}$-structures. 
Given a finite group, Blumberg and Hill \cite{BH15} studied homotopy-coherent commutative multiplicative structures in equivariant settings and found that, unlike in the nonequivariant setting, these so-called ``$N_\infty$-structures'' were no longer unique up to a reasonable notion of equivalence. 
In fact, the homotopy type of $N_\infty$-operads is completely described by \emph{transfer systems}, which are combinatorial objects, see e.g. \cite{Rub21} and \cite{BBR}. 

In a more category-theoretic language, a transfer system on a poset $\C$ is a wide subcategory of $\C$ closed under pullbacks. 
When $\C=Sub(G)$, the subgroup lattice of a finite group $G$, one recovers precisely the notion from \cite{BBR}. 
Subsequent work has not only linked transfer systems to well-established combinatorial methods, but has also showed that transfer systems can be identified as the acyclic fibrations of model category structures \cite{FOOQW} \cite{BOOR23}. 
Since the data of a model structure can be entirely determined by its classes of {weak equivalences} $\W$ and {acyclic fibrations} $\AF$, the model category information on a lattice is given by $\W$ and a transfer system.

Creating one model structure from an existing one can be a rather difficult process, and \emph{Bousfield localization} provides one way of doing just this. 
It is a powerful tool in homotopy theory used to enlarge the weak equivalence class of a model category.
The aim of localization is to formally invert morphisms in its associated homotopy category.
In this paper, we show how the transfer system given by a model structure on a lattice changes under left and right Bousfield localization. 
Furthermore, we describe an explicit set of morphisms, which we call \emph{Golden Arrows} $\Gamma_f$, that generates the new acyclic fibration set by adding $\Gamma_f$ to the old $\AF$ as a transfer system. 
In other words, the new acyclic fibrations are given by $\left< \AF \cup \Gamma_f \right>$, where $\left< M \right>$ denotes the smallest transfer system containing the set $M$. 
More precisely, we prove the following statement, which appears as Theorem \ref{thm:golden}. 

\begin{theorem*}
    For any model structure $(\W,\AF)$ on a lattice and an indecomposable morphism $f$, right Bousfield localization at $f$ coincides with the model structure produced in the following way. 
    \begin{itemize}
        \item The new weak equivalences are given by iteratively closing $\W \cup \left<\{f\}\right>$ under a two-out-of-three property and transfer system operations.
        \item The new acyclic fibrations are given by the set $\left< \AF \cup \Gamma_f \right>$. 
    \end{itemize}
\end{theorem*}

This further generalizes a theorem in \cite[Section 5]{BOOR23}, which characterized Bousfield localization on the finite total orders $[n]$. 
Dually, we can give an analogous characterization of left Bousfield localization using the notion of co-transfer systems. 
Although co-transfer systems have been less studied than their dual counterparts, this contributes to a more comprehensive description of Bousfield localization overall, which in turn contributes to studying their interactions with other structures.

Another topic that this article addresses is the effect of Bousfield localization on special kinds of transfer systems known as \emph{saturated transfer systems}.
In recent work, saturated transfer systems have been strongly linked to equivariant linear isometries operads \cite{Rub21}.
Furthermore, crucial constructions of commutative $G$-spectra rely on the saturated hull of a transfer system rather than the just transfer system itself \cite{BH21} \cite{BHKKNOPST}. 
In our context, we investigate the following. 
While for total orders $[n]$, every model category structure can be written as a sequence of left and right localizations of the trivial model \cite{BOOR23}, one soon observes this to be false in general, and this article provides many counterexamples. 
Yet, saturation helps us to better understand this observation. 
Not only does localization preserve non-saturation, but we can say the following in Theorem \ref{thm:saturatedzigzag}, which provides another initially unexpected bridge between equivariant homotopy theory and model categories.

\begin{theorem*}
Let $(\W,\AF)$ be a model structure on $[m] \times [n]$ where $\AF$ is a saturated transfer system. Then this model structure can be written as applying a sequence of left and right localizations to the trivial model structure.
\end{theorem*}

We would like to note that while our work is presented for finite lattices, we expect many of the concepts to generalize to arbitrary lattices. 

This paper is organized as follows.
We will gather the necessary terminology in Section \ref{sec:background} where we highlight some general properties of model category structures on posets as well as a summary of left and right Bousfield localization, before linking these to the language of transfer systems. With this in place, we move on to Section \ref{sec:rightloc}, where we present our construction on how transfer systems change under right Bousfield localization by adding just a minimal set of arrows, before presenting the dual constructions involving left localization and co-transfer systems in Section \ref{sec:dual_story}. 
Finally, in Section \ref{sec:saturation} we introduce the concept of saturation and examine its interaction with localization.

\section*{Acknowledgments}
This work originated as a group project in the 2024 AMS Mathematics Research Community ``Homotopical Combinatorics''. All authors gratefully acknowledge travel support from the AMS, and we would like to thank all organizers and fellow participants for creating an inspiring and supportive environment. 
AC was supported by UNAM Posdoctoral Program (POSDOC).
SG would like to thank the Krea University for further support.
SM would like to thank Raghav Malik for all the fruitful conversations on the subject matter.
CR thanks the University of Kent for further travel support, and the Isaac Newton Institute Cambridge for their support during the programme ``Equivariant Homotopy Theory in Context'' via EPSRC grant EP/Z000580/1.

Finally, we thank Devin Hensley for her involvement and enthusiasm during the week of our AMS MRC meeting.

\section{Background and conventions}\label{sec:background}

\subsection{Model Category Structures}\label{subsec:modelcats}
We assume that the reader is familiar with the basic notions related to model categories, see e.g. \cite[Chapter A.1]{BR20}. 
We will however reintroduce some definitions and properties that are central to the methods used in our paper. 

\begin{definition}\label{def1.1}
For any two morphisms $i \colon A \to B$ and $p \colon X \to Y$ in a category $\Cate$, we say that $i$ \emph{has the left lifting property (LLP) with respect to $p$}, or $p$ \emph{has the right lifting (RLP) property with respect to $i$},
if for all commutative squares of the form
\begin{equation}\label{liftdiagram}
\begin{tikzcd}
A \arrow[r] \arrow[d, "i"']         & X \arrow[d, "p"] \\
B \arrow[r] \arrow[ru, "h", dashed] & Y               
\end{tikzcd}    
\end{equation}
there exists a lift $h \colon B \to X$ which makes the resulting diagram commute.
If $i$ lifts on the left of $p$ we write $i \boxslash p$. 
For any class $\mathcal{S}$ of morphisms in $\Cate$ we write
\begin{align*}
\mathcal{S}^{\boxslash}& = \{g \in \operatorname{Mor}(\Cate) \mid f \boxslash g \text{ for all } f \in \mathcal{S} \},\\
{}^{\boxslash} \mathcal{S} &= \{f \in \operatorname{Mor}(\Cate) \mid f \boxslash g \text{ for all } g \in \mathcal{S} \}.
\end{align*}

\end{definition}

Note that $\mathcal{S} \subseteq {}^\boxslash \mathcal{T}$ if and only if $\mathcal{T} \subseteq \mathcal{S}^\boxslash$. 
We write $\mathcal{S}\boxslash \mathcal{T}$ when this holds.

\begin{definition}\label{D:2_out_of_3}
    Given any collection of morphisms in a category $M \subset \operatorname{Mor}(\C)$, we say that $M$ has the \emph{two-out-of-three property} if for any pair of composable morphisms $f,g$ in $\operatorname{Mor}(\C)$ where any two of $f,g$, or $g\circ f$ are in $M$, then the third must also be in $M$. 
    \end{definition}
    Diagrammatically, this definition is say if any two arrows in the diagram below are in $M$, then they must all be in $M$.
    
\[\begin{tikzcd}
	& \bullet \\
	\bullet && \bullet
	\arrow["g", from=1-2, to=2-3]
	\arrow["f", from=2-1, to=1-2]
	\arrow["{g \circ f}"', from=2-1, to=2-3]
\end{tikzcd}\]

\begin{definition}\label{def1.3}
A \emph{model category} is a category $\Cate$ equipped with three distinguished classes of morphisms, namely
\begin{itemize}
    \item weak equivalences -- $\W$ -- whose elements we will represent as $\xymatrix{X \ar[r]^{\sim} & Y}$,
    \item fibrations -- $\F$ -- whose elements we will represent as $\xymatrix{X \ar@{->>}[r] & Y}$,
    \item cofibrations -- $\Cof$ -- whose elements we will represent as $\xymatrix{X \ar@{^(->}[r] & Y}$,
\end{itemize}
each of which is closed under composition. A morphism in $\AF:=\W \cap \F$ (resp., $\AC:=\W \cap \Cof$) is said to be an \emph{acyclic fibration} (resp., \emph{acyclic cofibration}). These distinguished classes of morphisms and the category $\Cate$ are required to satisfy the following axioms.
\begin{enumerate}[align=left]
    \item[MC1)] The category $\Cate$ has all finite limits and colimits. 
    \item[MC2)] The class $\W$ satisfies the two-out-of-three property. 
    \item[MC3)] The three distinguished classes of morphisms are closed under retracts in the arrow category.
    \item[MC4)] Given a commutative diagram of the form (\ref{liftdiagram}), a lift exists when either $i$ is a cofibration and $p$ is an acyclic fibration, or when $i$ is an acyclic cofibration and $p$ is a fibration.
    \item[MC5)] Each morphism $f$ in $\Cate$ an be factored in two ways:
    \begin{enumerate}
        \item[1.] $f=p\circ i$, where $i$ is a cofibration and $p$ is an acyclic fibration.
        \item[2.] $f=p\circ i$, where $p$ is a fibration and $i$ is an acyclic cofibration.
    \end{enumerate}
\end{enumerate}
If objects $X$ and $Y$ lie in the same weak equivalence class then we shall write $X \simeq Y$. A specific choice of $\W, \F ,\Cof \subset \operatorname{Mor} (\Cate)$ is called a \textit{model structure} on $\Cate$.
\end{definition}
It is worth noting that the data of a model structure on a category $\Cate$ is overdetermined by sets of morphisms $\W ,\F ,\Cof, \AF$ and $\AC$ satisfying the conditions above.
In fact, only knowing e.g. $\W$ and $\AF$ or knowing $\W$ and $\AC$ suffices to determine the entire model structure.
We will adopt the convention of using the weak equivalences and acyclic fibrations, $\W$ and $\AF$, to describe a model category.
Hence, we write $(\W,\AF)_\Cate$ to denote a specific model structure determined by $\W$ and $\AF$ on a category $\Cate$, dropping the subscript when there is no risk of confusion. 

Let $\Cate$ be a complete and cocomplete category equipped with a model structure. 
If only the isomorphisms in $\Cate$ are weak equivalences, we say that the model structure is \emph{trivial}.

\subsection{Model Structures on Lattices and Transfer Systems}
Recall that a poset $\mathcal{P}=(P,\leq)$ is a category where $\mathrm{Ob}(\mathcal{P})=P$ and for any two objects $x,y$ there 
is at most one morphism, more precisely: 
$$\mathrm{Hom}_\C(x,y)=\left\lbrace\begin{array}{cc}
    \{\ast\} & \mbox{if }x\leq y, \\
    \emptyset &  \mbox{else}.
\end{array}\right.$$
In this category, given two objects $x,y$, their product, if it exists, is an object $z$ such that $z\leq x,y$ and for any other object 
$w\leq x,y$ we have that $w\leq z$. 
This means, $z$ is the \textit{meet} (the infimum) $x\wedge y$. In the same fashion, the coproduct is the 
\textit{join} (the supremum) $x\vee y$. A poset is a \textit{lattice} if for any two elements their join and their meet exist.

When we consider a model structure on a lattice $\mathcal{P}$, the axiom MC3 in Definition \ref{def1.3}, tells us that the three classes of morphisms must be closed under retracts in the arrow category.
In lattices there are no non-trivial retracts, as in a diagram
\begin{equation*}
    \xymatrix{
    x \ar@{->}[r] \ar@{->}[d] & z \ar@{->}[r] \ar@{->}[d] & x\ar@{->}[d]\\
    y \ar@{->}[r] & w \ar@{->}[r] & y,}
\end{equation*}
we have that $x\leq z\leq x$ and $y\leq w\leq y$ and therefore $x=z$ and $w=y$.
On a lattice, we furthermore have a stronger axiom than that of MC2 in Definition \ref{def1.3}. 
We call a class $\mathcal{S}$ of morphisms is \textit{decomposable} if for any morphism $f=g \circ h$ in $\mathcal{S}$, then both $g$ and $h$ are in $\mathcal{S}$, see \cite[Lemma 3.15]{MORSVZ2} or \cite[Proposition 1.8]{DZ21}.

\begin{lemma}
Let $\mathcal{P}$ be a lattice. 
Then the weak equivalences of any model structure on $\mathcal{P}$ are decomposable.
\end{lemma}

Therefore, we will be able to fully describe a weak equivalence set $\W$ in terms of only its indecomposable arrows.

Next, we turn to the acyclic fibrations. 
It turns out, the morphisms that define a set of acyclic fibrations  satisfy the following properties which also define the notion of a \emph{transfer system}.

\begin{definition}\label{def:transfer}
     A \emph{transfer system} in a category $\Cate$ is a wide subcategory of $\Cate$ which is closed under pullbacks by arbitrary morphisms in $\Cate$.
\end{definition}

To better understand this categorical definition, consider the subgroup lattice of a finite group \(G\), denoted by \(\mathrm{Sub}(G)\). (For simplicity of presentation, we assume our group to be abelian.)
We consider \(\mathrm{Sub}(G)\) as a category where the objects are subgroups of $G$, and the morphism $H \rightarrow K$ exists if $H$ is a subgroup of $K$.
For example, for a cyclic group of order $p^nq^m$ where $p \neq q$ are primes, the subgroup lattice $Sub(C_{p^nq^m})$ is a product of total orders $[n]\times[m]$.

In this situation, Definition \ref{def:transfer} becomes the following.

\begin{definition}
    A transfer system is a partial order on \(\mathrm{Sub}(G)\) that refines inclusion, and is closed under conjugation and restriction.
    More explicitly, a set of morphisms $M \subset \Mor(Sub(G))$ is called a transfer system if the following hold. 
    \begin{itemize}
        \item If $H < K \in M$ and $K< L \in M$, then $H<L \in M$ (composition).
        \item If $H<K \in M$ and $L \leq G$, then $H \cap L < K \cap L \in M$. (restriction).
    \end{itemize}
\end{definition}

\begin{convention}
Note that in our diagrammatic representations of grid-type lattices $[n] \times [m]$, we use a coordinate-type depiction, i.e. the origin $(0,0)$ is located at the bottom left of the grid, and $(n,m)$ is the top right vertex. 
\end{convention} 

\begin{example}\label{ex:1by1}
   In Figure \ref{fig:ten} we depict all ten possible transfer systems on the lattice \([1] \times [1]\).

\begin{figure}[H]
\begin{tikzpicture}[scale=1.1]
 \draw[blue!30, rounded corners] (0.6, 0.6) rectangle (2.4,2.4); 
 \foreach \x in {1,2}
 \foreach \y in {1,2}
    \fill (\x,\y) circle (1mm);
\end{tikzpicture}
\hspace{.5cm}
\begin{tikzpicture}[scale=1.1]
 \draw[blue!30, rounded corners] (0.6, 0.6) rectangle (2.4,2.4); 
 \foreach \x in {1,2}
 \foreach \y in {1,2}
    \fill (\x,\y) circle (1mm);
    \node at (1,2) (Sf) {}; 
    \node at (2,2) (Tf) {}; 
    \node at (2,1) (Tg1) {};
    \node at (1,1) (Sg1) {};
    \draw[line width=.5mm, black, -{stealth}] (Sg1) edge node[above]{} (Tg1);
\end{tikzpicture}
\hspace{.5cm}
\begin{tikzpicture}[scale=1.1]
 \draw[blue!30, rounded corners] (0.6, 0.6) rectangle (2.4,2.4); 
 \foreach \x in {1,2}
 \foreach \y in {1,2}
    \fill (\x,\y) circle (1mm);
    \node at (1,2) (Sf) {}; 
    \node at (2,2) (Tf) {}; 
    \node at (2,1) (Tg1) {};
    \node at (1,1) (Sg1) {};
    \draw[line width=.5mm, black, -{stealth}] (Sg1) edge node[above]{} (Sf);
\end{tikzpicture}
\hspace{.5cm}
\begin{tikzpicture}[scale=1.1]
 \draw[blue!30, rounded corners] (0.6, 0.6) rectangle (2.4,2.4); 
 \foreach \x in {1,2}
 \foreach \y in {1,2}
    \fill (\x,\y) circle (1mm);
    \node at (1,2) (Sf) {}; 
    \node at (2,2) (Tf) {}; 
    \node at (2,1) (Tg1) {};
    \node at (1,1) (Sg1) {};
    \draw[line width=.5mm, black, -{stealth}] (Sf) edge node[above]{} (Tf);
    \draw[line width=.5mm, black, -{stealth}] (Sg1) edge node[above]{} (Tg1);
\end{tikzpicture}
\hspace{.5cm}
\begin{tikzpicture}[scale=1.1]
 \draw[blue!30, rounded corners] (0.6, 0.6) rectangle (2.4,2.4); 
 \foreach \x in {1,2}
 \foreach \y in {1,2}
    \fill (\x,\y) circle (1mm);
    \node at (1,2) (Sf) {}; 
    \node at (2,2) (Tf) {}; 
    \node at (2,1) (Tg1) {};
    \node at (1,1) (Sg1) {};
    \draw[line width=.5mm, black, -{stealth}] (Sg1) edge node[above]{} (Sf);
     \draw[line width=.5mm, black, -{stealth}] (Sg1) edge node[above]{} (Tg1);
\end{tikzpicture}\\
\vspace{.5cm}
\begin{tikzpicture}[scale=1.1]
 \draw[blue!30, rounded corners] (0.6, 0.6) rectangle (2.4,2.4); 
 \foreach \x in {1,2}
 \foreach \y in {1,2}
    \fill (\x,\y) circle (1mm);
    \node at (1,2) (Sf) {}; 
    \node at (2,2) (Tf) {}; 
    \node at (2,1) (Tg1) {};
    \node at (1,1) (Sg1) {};
    \draw[line width=.5mm, black, -{stealth}] (Sg1) edge node[above]{} (Sf);
    \draw[line width=.5mm, black, -{stealth}] (Tg1) edge node[above]{} (Tf);
\end{tikzpicture}
\hspace{.5cm}
\begin{tikzpicture}[scale=1.1]
 \draw[blue!30, rounded corners] (0.6, 0.6) rectangle (2.4,2.4); 
 \foreach \x in {1,2}
 \foreach \y in {1,2}
    \fill (\x,\y) circle (1mm);
    \node at (1,2) (Sf) {}; 
    \node at (2,2) (Tf) {}; 
    \node at (2,1) (Tg1) {};
    \node at (1,1) (Sg1) {};
    \draw[line width=.5mm, black, -{stealth}] (Sg1) edge node[above]{} (Sf);
    \draw[line width=.5mm, black, -{stealth}] (Sg1) edge node[above]{} (Tf);
     \draw[line width=.5mm, black, -{stealth}] (Sg1) edge node[above]{} (Tg1);
\end{tikzpicture}
\hspace{.5cm}
\begin{tikzpicture}[scale=1.1]
 \draw[blue!30, rounded corners] (0.6, 0.6) rectangle (2.4,2.4); 
 \foreach \x in {1,2}
 \foreach \y in {1,2}
    \fill (\x,\y) circle (1mm);
    \node at (1,2) (Sf) {}; 
    \node at (2,2) (Tf) {}; 
    \node at (2,1) (Tg1) {};
    \node at (1,1) (Sg1) {};
    \draw[line width=.5mm, black, -{stealth}] (Sg1) edge node[above]{} (Sf);
    \draw[line width=.5mm, black, -{stealth}] (Tg1) edge node[above]{} (Tf);
    \draw[line width=.5mm, black, -{stealth}] (Sg1) edge node[above]{} (Tf);
     \draw[line width=.5mm, black, -{stealth}] (Sg1) edge node[above]{} (Tg1);
\end{tikzpicture}
\hspace{.5cm}
\begin{tikzpicture}[scale=1.1]
 \draw[blue!30, rounded corners] (0.6, 0.6) rectangle (2.4,2.4); 
 \foreach \x in {1,2}
 \foreach \y in {1,2}
    \fill (\x,\y) circle (1mm);
    \node at (1,2) (Sf) {}; 
    \node at (2,2) (Tf) {}; 
    \node at (2,1) (Tg1) {};
    \node at (1,1) (Sg1) {};
    \draw[line width=.5mm, black, -{stealth}] (Sg1) edge node[above]{} (Sf);
    \draw[line width=.5mm, black, -{stealth}] (Sg1) edge node[above]{} (Tf);
    \draw[line width=.5mm, black, -{stealth}] (Sf) edge node[above]{} (Tf);
    \draw[line width=.5mm, black, -{stealth}] (Sg1) edge node[above]{} (Tg1);
\end{tikzpicture}
\hspace{.5cm}
\begin{tikzpicture}[scale=1.1]
 \draw[blue!30, rounded corners] (0.6, 0.6) rectangle (2.4,2.4); 
 \foreach \x in {1,2}
 \foreach \y in {1,2}
    \fill (\x,\y) circle (1mm);
    \node at (1,2) (Sf) {}; 
    \node at (2,2) (Tf) {}; 
    \node at (2,1) (Tg1) {};
    \node at (1,1) (Sg1) {};
    \draw[line width=.5mm, black, -{stealth}] (Sg1) edge node[above]{} (Sf);
    \draw[line width=.5mm, black, -{stealth}] (Tg1) edge node[above]{} (Tf);
    \draw[line width=.5mm, black, -{stealth}] (Sg1) edge node[above]{} (Tf);
    \draw[line width=.5mm, black, -{stealth}] (Sg1) edge node[above]{} (Tg1);
    \draw[line width=.5mm, black, -{stealth}] (Sf) edge node[above]{} (Tf);
\end{tikzpicture}
\caption{Transfer Systems on \([1] \times [1]\)}
\label{fig:ten}
\end{figure}
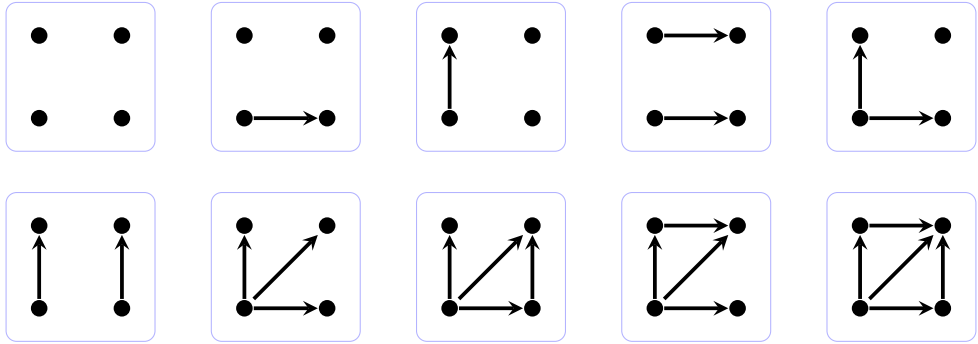

\end{example}

A key ingredient of our work is the following observation, which holds because acyclic fibrations are closed under pullbacks.

\begin{proposition}\label{prop:transferaf}
Let $P$ be a lattice equipped with a model structure. Then the acyclic fibrations on $P$ form a transfer system. 
\end{proposition}

\begin{remark}\label{rem:tmin}
    Even though the collection of $\AF$s is a transfer system, the converse is not true -- not every transfer system on a given poset appears as the $\AF$ of a valid model structure.
    For example, on the poset $[1] \times [1]$, there is no model structure with $\W=\{ (0,0) \rightarrow (1,0) \}$ and trivial $\AF$, even though the trivial transfer system is, of course, a transfer system contained in $\W$, see \cite[Remark 3.15]{BOOR23}.
\end{remark}

In fact, Remark \ref{rem:tmin} is an example of the following result from \cite[Theorem 4.20]{MORSVZ2},
which states that a transfer system arises as the $\AF$ of a model structure if and only if it contains a fixed minimal transfer system.

\begin{theorem}[MORSVZ]\label{thm:tmin}
Let $P$ be a finite lattice, and let $\W$ be set of weak equivalences of a model structure on $P$. 
Then there is a minimal transfer system $T_{min}$ depending on $\W$ such that the following are equivalent.
\begin{itemize}
\item $T$ is a transfer system with $T_{min} \subseteq T \subseteq \W$,
\item $(W, T)$ is a model structure, i.e. there is a model structure with weak equivalences $\W$ and $T=\AF.$
\end{itemize}
\end{theorem}

\begin{remark}\label{rem:tminconstruct}
The construction of $T_{min}$ can be made explicit.
Given a set of weak equivalences $\W$, we start with $K$ the largest co-transfer system contained in $\W$. (Here, a co-transfer system is a wide subcategory closed under pushouts, see Definition \ref{def:cotransfer}.) We consider $K$ to be the largest possible set of acyclic cofibrations in $\W$.
Then, $T_{min} = K^\boxslash \cap \W$, where $K^\boxslash$ denotes all morphisms with the RLP with respect to $K$.
For example, in Remark \ref{rem:tmin} one obtains $T_{min}=\W.$
\end{remark}

\subsection{Left and Right Bousfield Localization}\label{sec:bousfieldloc}
When given a model structure on a category, it is generally of interest to create other model structures on the same category by means of the original model structure. 
One strategy for doing this is to enlarge the class of weak equivalences.
Bousfield localizations provide a universal way of executing this strategy.

Starting with a model structure on a category, let us denote the \emph{left Bousfield localization} of a model structure $(\W,\AF)$ by 
\[
L_f(\W, \AF) = (L_f(\W), L_f(\AF))
\]
as well as write $L_f(\Cof),L_f(\F),L_f(\AC)$ for the respective cofibrations, fibrations and acyclic fibrations after left localization.
The set $L_f(\W)$ is the smallest weak equivalence set containing $f$ and $\W$ in some universal sense, while $L_f(\Cof):=\Cof$. 
Dually, the \emph{right localization} of $(\W,\AF)$ is denoted by  
\[
R_f(\W, \AF) = (R_f(\W), R_f(\AF)),
\]
and we have the sets $R_f(\Cof),R_f(\F)$ and $R_f(\AC)$. Again, $R_f(\W)$ is the smallest weak equivalence set containing $\W$ and $f$ and satisfying a universal property, and $R_f(\F):=\F$.
For a detailed treatment of Bousfield localization, see \cite[Chapter 7.1]{BR20}, but we would like to remark that on a lattice, the usual set-theoretic hurdles for the existence of a Bousfield localization become trivial.

Of course, when altering one class of morphisms in a model structure, one is typically interested in how this alteration affects the other classes of maps, and whether one can control this alteration. A basic consequence of the definitions is the following.
\begin{itemize}
\item $L_f(\W) \supseteq \W$,
\item $L_f(\Cof) = \Cof$ and $L_f(\AF) = \AF$,
\item $L_f(\AC)\supseteq \AC$,
\item $L_f(\F) \subseteq \F$.
\end{itemize}
For right Bousfield localization, dually, we have
\begin{itemize}
\item $R_f(\W) \supseteq \W$,
\item $R_f(\F) = \F$ and $R_f(\AC) = \AC$,
\item $R_f(\AF)\supseteq \AF$,
\item $R_f(\Cof) \subseteq \Cof$.
\end{itemize}
As described above, left localization does not change acyclic fibrations, so $L_f(\AF)=\AF$.
However, $R_f(\AF) \supset \AF$.
Therefore, our question now becomes, how does right localization affect the transfer system $\AF$?
In the case where the lattice is the total order lattice, $[n]$, this has been answered by \cite[Proposition 5.12]{BOOR23} and we recall their result here. 

\begin{proposition}\label{prop:totalordergoldenarrow}
Suppose that the poset $[n]$ possesses a model structure $\C$ where  for $0 \leq i < n$, $i \longrightarrow i+1$ is not a weak equivalence. 
Then the acyclic fibrations after right Bousfield localization at $i \rightarrow i+1$ are the acyclic fibrations of $\C$ with the addition of the arrows $m' \rightarrow j'$, where
\[i < j' \leq j.\]
Here, $m' \rightarrow i$ are acyclic fibrations in the old model structure, and $j$ is the largest number such that $i+1 \rightarrow j$ is a weak equivalence and $i+1 \rightarrow j+1$ is not. 

In particular, all the arrows of the form $i \rightarrow j'$, $i< j' \leq j$ are new acyclic fibrations.
\end{proposition}

Upon closer inspection, we note that in this proposition the only truly ``new'' morphism added to the acyclic fibrations is $i \rightarrow j$, and all other new acyclic fibrations are the result of adding this arrow to the old acyclic fibrations and closing the set under transfer system operations. We depict this in Figure \ref{fig:total}.

\begin{figure}[H]
\hspace{-1.2cm}
\begin{tikzpicture}[scale=1.1]
\draw[thick,dotted] (-0.5,2) -- (0.25, 2);
\draw[red!30, fill, rounded corners] (0.6, 1.6) rectangle (5.4, 2.4); 
\draw[red!30, fill, rounded corners] (5.6, 1.6) rectangle (8.4, 2.4);
\draw[thick, dotted] (8.75,2) -- (9.5, 2);

 \foreach \x in {1,2,3,4,5,6,7,8}
  \foreach \y in {2,2}
    \fill (\x,\y) circle (1mm);
   \node at (5,2) (Sf) {}; 
    \node at (8,2) (Tf) {}; 
    \node at (5,1.3) {${i}$} ;
     \node at (6,1.3) {$i+1$} ;
    \node at (8,1.3) {$j$} ;
    
 \draw[line width=.5mm, black, -{stealth}, bend left] (Sf) edge node[above]{} (Tf);
\end{tikzpicture}
\caption{Right localization at the arrow $i \rightarrow j$ in $[n]$. The shaded areas depict the weak equivalence classes.}
\label{fig:total}
\end{figure}
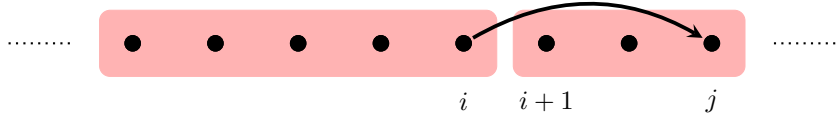

The main result of the next section will be generalizing this result to arbitrary finite lattices.

\section{An explicit characterization of Bousfield localization on finite lattices}\label{sec:rightloc}

We know that the data of a model structure is encoded entirely by its weak equivalences $\W$ together with its acyclic fibrations $\AF$, the latter forming a transfer system. 
Right Bousfield localization enlarges both $\W$ and $\AF$, and the goal of this section is to provide explicit constructions to make this precise, which we state in Theorem \ref{thm:golden}.
Moving forward, we will specifically focus on right localizing at \emph{short edges} as in the following definition.

 \begin{definition}
 We say that a morphism $f$ in a lattice is \emph{short} if it is indecomposable, i.e. if $f=f_2\circ f_1$ then either $f_2$ or $f_1$ is the identity. 
 \end{definition}

By the universal property of Bousfield localization, every Bousfield localization can be written as a sequence of localizations along short arrows. Therefore, it will suffice for our purpose to understand those. 

\begin{definition}\label{def:rubin}
    Given any collection of morphisms in a category, $M \subset \operatorname{Mor}(\C),$
    let the \emph{transfer system closure} of $M$, denoted $\left< M \right>$, be the smallest transfer system on $\C$ containing $M$.
\end{definition}

In general, $\left< M \right>$ is the smallest subcategory of $\C$ containing $M$ that is closed under pullbacks. In practice, the order of closing under composition and pullbacks can be fixed as in the construction below.

\begin{construction}
For poset categories $\C$ there is an explicit procedure for constructing $\langle M \rangle$, i.e. the smallest transfer system containing the set $M$, see \cite[Lemma 3.6]{MORSVZ}. We give the procedure below. The construction involves expanding the collection $M$ in two stages $M_1, M_2$ to build the smallest transfer system containing $M$.

\begin{itemize}
    \item Restriction: $M_{1}\coloneqq \{ K \wedge L \to L  \mid (K \to H) \in M, (L \to H) \in \C \}$,
    \item Composition: $M_{2}\colon= \{f_{1}\circ f_{2} \circ \ldots \circ f_{n} \mid f_{1}, \ldots, f_{n}\in M_1\}$.
\end{itemize}

The resulting collection $M_2$ is the transfer system closure $\langle M \rangle$ that we seek. In words, $M_1$ expands $M$ by addding all morphisms that are pullbacks of any morphisms in $M$ along any map in the category. Finally, $M_2$ expands $M_1$ by addding any morphisms that are compositions of finitely many morphisms in $M_1$.

\end{construction}

Now we describe the general construction for enlarging the acyclic fibrations of a model structure by a set $M$. The goal of this section is to produce a set $\Gamma_f$ such that the right localization $R_f(\W,\AF)$ of a model structure $(\W,\AF)$ is given by enlarging $\AF$ by $\Gamma_f$.

\begin{definition}\label{def:model}
    Given a model structure $(\W,\AF)$ on a lattice and a set of morphisms $M$, we define the \textit{model structure closure with respect to $M$} as 
    \[ M(\W,\AF) = (R_M(\W),\left<\AF \cup M \right>). \] 
\end{definition}

Proposition \ref{prop:M(W,AF)_is_a_model_str} will show that $M(\W,\AF)$ is indeed a model structure, which justifies the name.

\begin{proposition}\label{prop:M(W,AF)_is_a_model_str}
    Let $(\W,\AF)$ be a model structure on a lattice $P$, then $M(\W,\AF)$ is also a model structure on $P$.
\end{proposition}
\begin{proof}
First, we observe that by definition, $\left<\AF \cup M \right> \subseteq R_M(\W)$.
As $(\W, \AF)$ is assumed to be a model structure, we know that $\AF$ is a transfer system with 
\[
T_{min} \subseteq \AF \subseteq \W,
\]
see Theorem \ref{thm:tmin}. We therefore have to show that $T'_{min} \subseteq \left<\AF \cup M \right>$, where $T'_{min}$ is the minimal transfer system associated with $\W_M$. 
When forming $M(\W,\AF)=(\W_M,\left<\AF \cup M \right>)$, the size of the weak equivalence set increases, therefore by the construction given in Remark \ref{rem:tminconstruct}, the size of the minimal transfer system decreases.
 Since $\left<\AF \cup M \right> \supseteq \AF,$ we know that 
    $${T'}_{min} \subset {T}_{min} \subset \AF \subset \left<\AF \cup M \right>.$$
\end{proof}

In order to understand $M(\W,\AF)=(R_M(\W), \left<AF\cup M \right>)$,
let us now describe $R_M(\W)$ after performing a right localization of a model structure $(\W,\AF)$ at the set $M$. 
Firstly, note that in any model structure, short edges have to satisfy the following properties in order for the factorization axiom MC5 to hold.
\begin{itemize}
\item Every short edge that is \emph{not} a weak equivalence has to be a fibration \emph{and} a cofibration.
\item Every short edge that is a weak equivalence has to be a fibration \emph{or} a cofibration, but never both.
\end{itemize}
In addition to this, we recall that (acyclic) fibrations are closed under pullbacks, and (acyclic) cofibrations are closed under pushouts. 

Using this information, we see that in a model category all short edges in $\W$ either must have all their pushouts or all their pullbacks contained in $\W$ as well.
This is not sufficient to characterize whether a decomposable subcategory  arises as the weak equivalences of a model structure on a finite lattice - one has to add a further condition.
The following is \cite[Theorem 5.9]{MORSVZ2}

\begin{theorem}[MORSVZ2]\label{thm:legalw}
Let $\W$ be a decomposable subcategory of a finite lattice $P$. Then there is a model structure on $P$ with weak equivalences $\W$ if and only if the following two points are satisfied.
\begin{itemize}
\item For a short edge $\sigma \in \W$, either all pullbacks of $\sigma$ are in $W$ or all pushouts of $\sigma$ are in $\W$. 
\item For any factorization $f=\sigma_n \circ \sigma_{n-1} \circ \cdots \circ \sigma_1$ of a morphisms $f \in \W$ into short arrows, there is $k$ such that 
\begin{itemize}
\item for $\sigma_i$ with $i \leq k$, all pushouts of $\sigma_i$ are in $\W$,
\item for $\sigma_i$ with $i >k$, all pullbacks of $\sigma_i$ are in $\W$.
\end{itemize} 
\end{itemize}
\end{theorem}

Returning to right Bousfield localization, we know that the right localization $R_f(\W,\AF)$ of a model structure $(\W,\AF)$ at a short edge $f$ is characterized by the following.
\begin{itemize}
\item The weak equivalences $R_f(\W)$ contain $\W$ and all pullbacks of $f$. 
\item The weak equivalences $R_f(\W)$ satisfy the conditions of Theorem \ref{thm:legalw}.
\item $R_f(\W)$ is the smallest decomposable subcategory satisfying the previous two points. 
\end{itemize}

To see what arrows we have to add to $\W$ to obtain $R_f(\W)$, we need to understand the difference between a model structure and a pre-model structure, which is simply a model structure without the two-out-of-three axiom MC2, see e.g. \cite[4.7 and 4.8]{BOOR23} and \cite[Remark 4.6]{FOOQW}. 
The data of a pre-model structure is equivalent to two pairs
\[
({}^\boxslash T_1, T_1), ({}^\boxslash T_2, T_2)
\]
where $T_1 \subseteq T_2$ are transfer systems and ${}^\boxslash T_i$ is the set of maps that have the left lifting property with respect to $T_i$. A pair $({}^\boxslash T, T)$ with $T$ a transfer system is also known in the literature as a \emph{weak factorization system}. In the above, the transfer system $T_2$ plays the role of the fibrations, $T_1$ of the acyclic fibrations, ${}^\boxslash T_2$ the acyclic cofibrations and ${}^\boxslash T_1$ the cofibrations. The weak equivalences are defined as the set $T_1\circ {}^\boxslash T_2 $, and our pre-model structure is furthermore a model structure if and only if $T_1\circ {}^\boxslash T_2 $ satisfies the two-out-of-three axiom.

Given a model structure with acyclic fibrations $\AF$ and acyclic cofibrations $\AC$, every transfer system $T \supseteq \AF$ gives rise to a {pre-model structure}  with acyclic fibrations $T$, acyclic cofibrations $\AC$ and $W':=T \circ \AC$.  
The obstacle for this pre-model structure to actually be a model structure is that $\W'=T \circ \AC$ does not necessarily satisfy the two-out-of-three axiom.

For our right localization, this means the following. We are adding an arrow $f$ to the acyclic fibrations, therefore the new acyclic fibrations after right localization contain the transfer system $T:=\left< \AF \cup \{f\}\right>$. If for this choice of $T$, $W'=T \circ \AC$ satisfies the two-out-of-three axiom, then $W'=R_f(\W)$, and we have found our new weak equivalences. If not, then we have to add to $T$. Before we make this precise in Construction \ref{con:rf}, we illustrate the process with the following example.

\begin{example}\label{ex:goodexample}

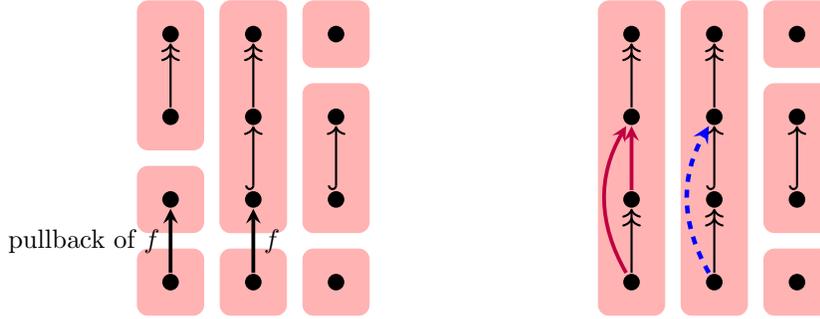
\begin{figure}[H]
    \begin{tikzpicture}[scale=1.1, hookarrow/.style={{Hooks[right]}->}]
    \draw[red!30, fill, rounded corners] (0.6,2.6) rectangle (1.4,4.4); 
  \draw[red!30, fill, rounded corners] (1.6,1.6) rectangle (2.4,4.4); 
    \draw[red!30, fill, rounded corners] (2.6,1.6) rectangle (3.4,3.4); 
    \draw[red!30, fill, rounded corners] (1.6,1.6) rectangle (2.4,4.4);
    \draw[red!30, fill, rounded corners] (0.6,0.6) rectangle (1.4,1.4); 
    \draw[red!30, fill, rounded corners] (1.6,1.4) rectangle (2.4,0.6); 
    \draw[red!30, fill, rounded corners] (2.6,1.4) rectangle (3.4,0.6);
    \draw[red!30, fill, rounded corners] (0.6,1.6) rectangle (1.4,2.4); 
    \draw[red!30, fill, rounded corners] (2.6,3.6) rectangle (3.4,4.4);  
    \foreach \x in {1,2,3}
    \foreach \y in {1,2,3,4}
    	\fill (\x,\y) circle (1mm);
    \node at (2,1) (Sf) {}; 
    \node at (2,2) (Tf) {}; 
    \node at (1,1) (Sfp) {}; 
    \node at (1,2) (Tfp) {}; 
    \node at (3,1) (A) {}; 
    \node at (3,2) (B) {}; 
      \node at (3,3) (K) {}; 
    \node at (4,1) (C) {}; 
    \node at (4,2) (D) {}; 
     \node at (1,3) (E) {}; 
    \node at (2,3) (F) {}; 
    \node at (1,4) (G) {}; 
    \node at (2,4) (H) {}; 
    \draw[line width=.5mm, black, -{stealth}] (Sf) edge node[right]{\color{black}$f$} (Tf);    
    \draw[line width=.5mm,  black, -{stealth}] (Sfp) edge node[left]{\color{black} pullback of $f$} (Tfp);
    \draw[thick, ->>] (E) edge node[above]{} (G);
     \draw[thick, ->>] (F) edge node[above]{} (H);
     \draw[thick, hookarrow](Tf) edge node[above]{} (F);
     \draw[thick, hookarrow](B) edge node[above]{} (K);
    \end{tikzpicture}
    \hspace{2cm}
   \begin{tikzpicture}[scale=1.1, hookarrow/.style={{Hooks[right]}->}]
    \draw[red!30, fill, rounded corners] (0.6,0.6) rectangle (1.4,4.4); 
    \draw[red!30, fill, rounded corners] (2.6,1.6) rectangle (3.4,3.4); 
    \draw[red!30, fill, rounded corners] (1.6,0.6) rectangle (2.4,4.4); 
    \draw[red!30, fill, rounded corners] (2.6,1.4) rectangle (3.4,0.6);
    \draw[red!30, fill, rounded corners] (2.6,3.6) rectangle (3.4,4.4);  
    \foreach \x in {1,2,3}
    \foreach \y in {1,2,3,4}
    	\fill (\x,\y) circle (1mm);
    \node at (2,1) (Sf) {}; 
    \node at (2,2) (Tf) {}; 
    \node at (1,1) (Sfp) {}; 
    \node at (1,2) (Tfp) {}; 
    \node at (3,1) (A) {}; 
    \node at (3,2) (B) {}; 
      \node at (3,3) (K) {}; 
    \node at (4,1) (C) {}; 
    \node at (4,2) (D) {}; 
     \node at (1,3) (E) {}; 
    \node at (2,3) (F) {}; 
    \node at (1,4) (G) {}; 
    \node at (2,4) (H) {}; 
    \draw[thick, ->>] (Sf) edge node[above]{} (Tf);
    \draw[thick, ->>] (Sfp) edge node[above]{} (Tfp);
    \draw[thick, ->>] (E) edge node[above]{} (G);
     \draw[thick, ->>] (F) edge node[above]{} (H);
     \draw[thick, hookarrow](Tf) edge node[above]{} (F);
     \draw[thick, hookarrow](B) edge node[above]{} (K);
     
     \draw[line width=.6mm, blue, -{stealth}, dashed] (Sf) edge[bend left] node[above]{} (F);
    \draw[line width=.5mm, purple, -{stealth}] (Tfp) edge node[left]{\color{black} } (E);
      \draw[line width=.5mm, purple, -{stealth}] (Sfp) edge[bend left] node[above]{} (E);
    \end{tikzpicture}
    \caption{Example of a right localization at the morphisms $f$. The shaded regions denote the respective weak equivalence sets.}
    \label{fig:goodpicture}
    \end{figure}

In Figure \ref{fig:goodpicture}, the first image denotes a model structure on the lattice $[2] \times [3]$, where the weak equivalences $\W$ are indicated by the shaded red region, and we would like to right localize at the short arrow $f$. We see that while all short arrows in $\W':=\W \cup \left<f\right>$ either have all their pushouts or pullbacks contained in $\W'$, this $\W'$ does not satisfy the second point of Theorem \ref{thm:legalw}.

For example, in our picture the arrows $(1,0) \rightarrow (1,1)$ and $(1,1) \rightarrow (1,2)$ are both in $T\circ \AC$ for $T=\left<\AF \cup \{f\}\right>$ as they are either in $T$ or $\AC$ themselves, but their composite $(1,0) \rightarrow (1,2)$ (the dashed blue arrow) is not in $T \circ \AC$. Therefore, $T \circ \AC$ is not closed under the two-out-of-three property.

To rectify this failure, we need to close $\left<\AF \cup \{f\}\right> \circ \AC$ under the two-out-of-three property. 
In our picture, this amounts to adding the arrow $(1,0) \rightarrow (1,2)$ to the new acyclic fibrations and therefore to $\W'$. 
As a consequence, we also have to add the pullbacks of any new arrows we added to $\W'$, and again any arrows that arise from applying the two-out-of-three closure to the result. 

In our second picture, those latter arrows are depicted by the purple arrows $(0,1) \rightarrow (0,2)$ and $(0,0) \rightarrow (0,2)$. We now see that all colored arrows in the second picture together with the original $\W$ form a decomposable subcategory containing $f$ and its pullbacks as well as satisfying Theorem \ref{thm:legalw}. As it is also the smallest such subcategory by construction, we see that we have indeed constructed $R_f(\W)$.

\end{example}

In general, we have the following process.

\begin{construction}\label{con:rf}
In order to obtain the right localization $R_f(\W)$ of the weak equivalence set $\W$, we perform these recursive steps.
\begin{enumerate}
\item $\AF_{(0)} := \AF, \,\,\, S_{(0)}:=\{f\}$,
\item $\AF_{(n)} := \left< \AF_{(n-1)} \cup S_{(n-1)} \right>$,
\item Take $\W_{(n)}$ to be the closure of $\AF_{(n)} \circ \AC$ under the two-out-of-three property.
\item Set $S_{(n)}$ to be the set of arrows in $\W_{(n)} \backslash \W_{(n-1)}$ and continue with Step (2).
\end{enumerate}
We know the process terminates or stabilizes, i.e. there is an $N$ with $W_{N}=W_{N+k}$ for all $k$ as we are in a finite lattice, and so we obtain a $\W_{(N)}$ satisfying Theorem \ref{thm:legalw}, and thus, $\W_{(N)}=R_f(\W)$. 
\end{construction} 
\vspace{-.4cm}
While this does not give a self-contained closed formula for $R_f(\W)$ involving $f$ and $\W$ only, it gives a straightforward and reliable procedure for obtaining the correct weak equivalences in examples. Note that a dual procedure exists for left localization.

For the rest of the section we will fix a model structure $(\W,\AF)$ on a lattice $P$, and a morphism $f$ whose source and target objects are denoted by $s(f)$, and $t(f)$, respectively. 
The goal is to define a set of {maximal arrows} $\Gamma_f$ called \emph{Golden Arrows} depending on $f$ and $\W$, and we start by defining the targets of said arrows.

Denote the weak equivalence class of an object $x$ in a model structure with weak equivalences $\W$ by $[x]_\W$.
Let $\sigma: s(\sigma) \rightarrow t(\sigma)$ be any morphisms in our lattice.
Define the set $T(\sigma)$ to be the maximal objects in the weak equivalence class of $t(\sigma)$, namely
\begin{eqnarray}
T(\sigma) & = & \{  y \in [t(\sigma)]_{\W} \,\vert\, \text{ there is no $z \in [t(\sigma)]_{\W}$ with $z >y$} \nonumber \} \\
 & = & \{  y \in [t(\sigma)]_{\W} \,\vert\, \text{ if $z \in [t(\sigma)]_{\W}$ and $y \leq z$ then $y=z$} \nonumber \}
\end{eqnarray}

 For the set $S(\sigma)$ take all $y \in [s(\sigma)]_{\W}$ such that $y$ is maximal among the objects in its weak-equivalence class before localization, and such that there is an arrow from $y$ to some element in $T(\sigma)$. Formally,
 \[
S(\sigma)= \{  y \in [t(\sigma)]_{\W} \,\vert\, y \leq t \text{ for some $t \in T(\sigma)$, and if $z \in [y]_{\W}$ and $y \leq z \leq t$ then $y=z$}\} 
\]

With this having been defined, consider the sets $S(\sigma)$ and $T(\sigma)$ for every short arrow $\sigma$ in $R_f(\W) \setminus \W$, i.e. short arrows that become weak equivalences after localizing but are not in $\W$ before localizing. As short arrows that are not weak equivalences are always fibrations, due to axiom (MC5), these arrows will need to be added to the set of new acyclic fibrations.
We now define the following set of maximal such arrows. 
\begin{definition}\label{def:goldenarrows}
    The \textit{Golden Arrows associated to the morphism $f$} are defined as the set 
    \[
\Gamma_f = \{    s \rightarrow t \, \vert \, s \in S(\sigma) , t \in T(\sigma) \,\,\mbox{for some short arrow $\sigma \in R_f(\W)\setminus \W$}  \}. 
\]
\end{definition} 

An example of the sources of $\Gamma_f$ is shown in Figure \ref{fig:starshaped}.

\begin{remark}\label{rem:goldensquare}
In particular, the set $\Gamma_f$ is constructed so that for every new weak equivalence $w \in R_f(\W) \backslash \W$ there is an element $g \in \Gamma_f$ such that $g \ge w$, i.e. for every new weak equivalence $w$ there is a golden arrow $g$ such that the following square commutes.

\[
\xymatrix{ s(g) \ar[r]^g & t(g) \\
s(w) \ar[u] \ar[r]^w & t(w) \ar[u]
}
\]
Furthermore, $g$ can be chosen such that the left vertical map is in $\W$.
\end{remark}

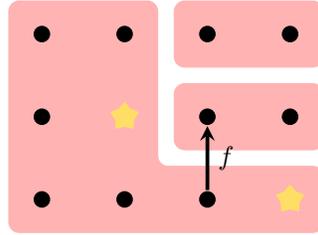
\begin{figure}[H]
    \begin{tikzpicture}[scale=1.1]
    \draw[red!30, fill, rounded corners] (2.6,2.6) rectangle (4.4,3.4); 
    \draw[red!30, fill, rounded corners] (2.6,1.6) rectangle (4.4,2.4); 
    \draw[red!30, fill, rounded corners] 
    (0.6, 0.6) -- 
    (4.4, 0.6) -- 
    (4.4, 1.4) -- 
    (2.4, 1.4) -- 
    (2.4, 3.4) -- 
    (0.6, 3.4) -- 
    cycle;
    \foreach \x in {1,2,3,4}
    \foreach \y in {1,2,3}
    	\fill (\x,\y) circle (1mm);
    \node[star, star points=5, minimum size=4mm, inner sep=2pt, fill=gold] at (4,1) {};
    \node[star, star points=5, minimum size=4mm, inner sep=2pt, fill=gold] at (2,2) {};
    \node at (3,1) (Sf) {}; 
    \node at (3,2) (Tf) {}; 
    \draw[line width=.5mm, black, -{stealth}] (Sf) edge node[right]{$f$} (Tf);
    \end{tikzpicture}
    \caption{The star-shaped nodes are the elements of $S(f)$ when right localizing along $f$, while the shaded regions depict the weak equivalences before localization.}
    \label{fig:starshaped}
    \end{figure}

\begin{example}
    Suppose $P$ is a $5 \times 5$ lattice  with the three weak equivalence classes and chosen morphism $f$ as shown in the left diagram below. Then the class $[t(f)]_{\W}$ is the upper right block. After localizing along $f$ to form $R_f(\W)$, there remains only one equivalence class of weak equivalences, as shown in the right diagram below. The object in the far upper right is the only object in $T(f)$, 
\end{example}
\begin{figure}[H]
    \begin{tikzpicture}[scale=1.1]
    \draw[red!30, fill, rounded corners] (3.6,2.6) rectangle (5.4,5.4);
    \draw[red!30, fill, rounded corners] (0.6,3.6) rectangle (3.4,5.4);
    \draw[red!30, fill, rounded corners] 
    (0.6, 0.6) -- 
    (5.4, 0.6) -- 
    (5.4, 2.4) -- 
    (3.4, 2.4) -- 
    (3.4, 3.4) -- 
    (0.6, 3.4) -- 
    cycle;
    \foreach \x in {1,2,3,4,5}
    \foreach \y in {1,2,3,4,5}
    	\fill (\x,\y) circle (1mm);
    \node at (3,3) (Sf) {}; 
    \node at (4,3) (Tf) {}; 
    \node at (5,5) (Tg) {};
    \node at (3,5) (Sg1) {};
    \node at (5,2) (Sg3) {};
    \draw[line width=.5mm, black, -{stealth}] (Sf) edge[bend left] node[above]{$f$} (Tf);
    \draw[line width=2.5pt, gold, -{stealth}, dashed] (Sg1) edge[bend left] node[above]{\color{black}$g_1$} (Tg);
    \draw[line width=2.5pt, gold, -{stealth}, dashed] (Sf) edge[bend left] node[left]{\color{black}$g_2$} (Tg);
    \draw[line width=2.5pt, gold, -{stealth}, dashed] (Sg3) edge[bend right] node[right]{\color{black}$g_3$} (Tg);
    \end{tikzpicture}
    \hspace{5em}
    \begin{tikzpicture}[scale=1.1]
    \draw[red!30, fill, rounded corners] (0.6, 0.6) rectangle (5.4,5.4);
    \foreach \x in {1,2,3}
    \foreach \y in {1,2,3,4,5}
    	\draw[very thick] (\x,\y) circle (1.7mm);
    \foreach \x in {4,5}
    \foreach \y in {1,2}
    	\draw[very thick] (\x,\y) circle (1.7mm);
    \foreach \x in {1,2,3,4,5}
    \foreach \y in {1,2,3,4,5}
    	\fill (\x,\y) circle (1mm);
    \node at (3,3) (S) {};
    \node at (4,3) (T) {};
    \draw[line width=.5mm, black, -{stealth}] (S) edge[bend left] node[above]{$f$} (T);
    \end{tikzpicture}
\end{figure}

The following example shows that is necessary to add Golden Arrows for all the pairs of connected components that are joined under right localization, not just for the source and target components of $f$.
\begin{example} 
Consider a model structure on $[2] \times [3]$ with weak equivalences as in Example
\ref{ex:goodexample},
and we right localize at the arrow $f:(1,0) \rightarrow (1,1)$. Definition \ref{def:goldenarrows} gives us
\[
\Gamma_f=\{(1,0)\rightarrow (1,3), \,(0,0) \rightarrow (0,1),\,(0,1) \rightarrow (0,3)\}.
\]
The part of $\Gamma_f$ arising from joining the source and target components of $f$ is $(1,0)\rightarrow (1,3)$. The second element of $\Gamma_f$, $(0,0) \rightarrow (0,1)$, is a pullback of $(1,0)\rightarrow (1,3)$ and therefore redundant. However, we do require $(0,1) \rightarrow (0,3)$: Example \ref{ex:goodexample} shows that $(0,1)\rightarrow (0,2)$ becomes an acylic fibration after right localization at $f$, but $(0,1) \rightarrow (0,2) \notin \left<\AF \cup\{(1,0)\rightarrow (1,3)\}\right>$.

\end{example}

\begin{example}\label{ex:squareexample}

Consider the following model structure on the lattice $[1] \times [1]$, where the weak equivalences are given by the shaded area and every morphisms is a fibration. 
This is the only possible model structure on $[1] \times [1]$ with this weak equivalence set, so
\[
\AF = \{ (0,0) \rightarrow (1,0), \,\, (0,0) \rightarrow (0,1) \},
\]
i.e. the transfer system of acyclic fibrations before localization consists of the bottom left corner of the square, depicted below: 
\begin{figure}[H]
\hspace{-1.2cm}
\begin{tikzpicture}[scale=1.1]
\draw[red!30, fill, rounded corners] (0.6, 0.6) rectangle (2.4, 1.4); 
\draw[red!30, fill, rounded corners] (1.6, 1.6) rectangle (2.4, 2.4); 
\draw[red!30, fill, rounded corners]  (0.6, 0.6)rectangle (1.4,2.4); 
 \foreach \x in {1,2}
 \foreach \y in {1,2}
    \fill (\x,\y) circle (1mm);
    \node at (1,2) (Sf) {}; 
    \node at (2,2) (Tf) {}; 
    \node at (2,1) (B) {};
    \node at (1,1) (C) {};
  \draw[line width=.5mm, black, -{stealth}] (Sf) edge node[above]{$f$} (Tf);
   \draw[thick, ->>, black] (C) edge (Sf);
    \draw[thick, ->>, black] (C) edge (B);
  
\end{tikzpicture}
\end{figure}
We look at what what happens to $\AF$ when localizing at $f$. 
We know that after right localization at this $f$, every morphism in $[1]\times [1]$ becomes a weak equivalence.
The Golden Arrows are denoted by $g_1$ and $g_2$ in the picture below.
\begin{figure}[H]
\begin{tikzpicture}[scale=1.1]
\draw[red!30, fill, rounded corners] (0.6, 0.6) rectangle (2.4, 1.4); 
\draw[red!30, fill, rounded corners] (1.6, 1.6) rectangle (2.4, 2.4); 
\draw[red!30, fill, rounded corners]  (0.6, 0.6)rectangle (1.4,2.4); 
 \foreach \x in {1,2}
 \foreach \y in {1,2}
    \fill (\x,\y) circle (1mm);
    \node at (1,2) (Sf) {}; 
    \node at (2,2) (Tf) {}; 
    \node at (2,1) (Tp) {};
      \node at (2,1) (B) {};
    \node at (1,1) (C) {};
   \draw[thick, ->>, black] (C) edge (Sf);
    \draw[thick, ->>, black] (C) edge (B);
    \draw[line width=.5mm, black, -{stealth}] (Sf) edge node[below]{$f$} (Tf);
    \draw[line width=2.5pt, gold, -{stealth}, dashed] (Sf) edge[bend left] node[above]{\color{black}$g_1$} (Tf);
    \draw[line width=2.5pt, gold, -{stealth},dashed] (Tp) edge[bend right] node[right]{\color{black}$g_2$} (Tf);
\end{tikzpicture}
\end{figure}
These Golden Arrows get incorporated into the set of acyclic fibrations in the new model structure, and since the acyclic fibrations need to form a transfer system we include the necessary arrows making this true.
Therefore, for this choice of $f$, we have
\[
R_f(\W,\AF)=(\mbox{all},\mbox{all}).
\]
\begin{figure}[H]
\hspace{-0.4cm}
\begin{tikzpicture}[scale=1.1]
 \draw[red!30, fill, rounded corners] (0.6, 0.6) rectangle (2.4,2.4);  
 \foreach \x in {1,2}
 \foreach \y in {1,2}
    \fill (\x,\y) circle (1mm);
    \node at (1,2) (Sf) {}; 
    \node at (2,2) (Tf) {}; 
    \node at (1,1) (Sp) {};
    \node at (2,1) (Tp) {};
    \draw[ thick, ->>] (Sf) edge node[above]{$f$} (Tf);
     \draw[thick, ->>] (Sp) edge node[above]{} (Tp);
     \draw[thick, ->>] (Sp) edge node[above]{} (Sf);
     \draw[thick, ->>] (Tp) edge node[above]{} (Tf);
     \draw[thick, ->>] (Tp) edge node[above]{} (Tf);
    \draw[thick, ->>] (Sp) edge node[above]{} (Tf);
\end{tikzpicture}
\end{figure}
\end{example}

Recall from Definition \ref{def:model} that for a set of morphisms $M$, $M(\W,\AF)=(R_M(\W),\left<\AF \cup M \right>)$ is the model structure obtained from adding $M$ to the set of acyclic fibrations, and that $\left< \AF\cup M \right>$ is the smallest transfer system containing $\AF$ as well as $M$. 

\begin{lemma}\label{lem:subset}
Let $(\W,\AF)$ be a model structure with fibrations $f$, and let $R_f(\W,\AF)=(R_f(\W), R_f(\AF))$ denote the right Bousfield localization of $(\W,\AF)$ at $f$. 
Then
    \[
    \left< \AF\cup \Gamma_f  \right> \subseteq \F \cap R_f(\W) = R_f(\AF).
    \] 
    In particular, the Golden Arrows $\Gamma_f$ are acyclic fibrations with respect to the right localized model structure at $f$.
\end{lemma}

\begin{proof}
    We begin by showing that $\Gamma_f \subset \F$. Indeed, if $g \in \Gamma_f$ then by the factorization axiom there must exists an acyclic cofibration $i$, and a fibration $p$, such that $g= p \circ i$. 
    However, by the maximality conditions that define $\Gamma_f$, any such factorization must have $i = \id_{s(g)}$, hence $p = g$, so $g \in \F$. 
    
    By definition of the Golden Arrows, the source and target of $g$ lie in the same weak equivalence class after localizing along $f$, so $g$ is also in $ R_f(\W)$.
    Hence $\AF \cup \Gamma_f \subseteq R_f(\AF)$.
    By taking closures under pullbacks and compositions as in Definition \ref{def:rubin}, we obtain 
    $$ \langle \AF \cup \Gamma_f \rangle\subseteq \langle R_f(\AF) \rangle = R_f(\AF)$$ as claimed. 
\end{proof}

As it would turn out, the reverse inclusion holds.

\begin{lemma}\label{lem:supset} Let $(\W,\AF)$ be a model structure with fibrations $\F$.
The acyclic fibrations after right localization are contained in the transfer system generated by $\AF$ and the Golden Arrows $\Gamma_f$, i.e.,
    \[
     \left< \AF\cup \Gamma_f  \right> \supseteq \F \cap R_f(\W) = R_f(\AF).
    \]
  
\end{lemma}

\begin{proof}
    Let $x\rightarrow y \in R_f(\AF)$. If $x\rightarrow y \in \AF$ then clearly it is in $R_f(\AF)$, and there is nothing to show. So assume $x\rightarrow y$ is not in $\AF$, i.e. $x \rightarrow y$ is an acyclic fibration after right localization but not before. Note that this implies in particular that $x\rightarrow y$ is a fibration before and after localization. \\
    We consider the following set.
    \[
    \Gamma_f \cap \{ g \,|\, x\rightarrow s(g)  \text{ is a weak equivalence before localizing, and } t(g) \geq y \}.
    \]
This set is non-empty because of Remark \ref{rem:goldensquare}, i.e. there is always a commutative diagram with $g \in \Gamma_f$ and $x \rightarrow s(g) \in \W$.

    \[
    \xymatrix{ s(g) \ar[r]^g &  t(g) \\
    x \ar[r]\ar[u]^{\sim} & y \ar[u]
    }
    \] 
     Let $g$ be an element of the set above, then $g$ is a golden arrow that is greater or equal to $x \rightarrow y$.
    Let $$B = \{ z \,\vert\, x \leq z<y,\,\, z <s(g) \},$$ and $b$ any maximal element of $B$ (i.e. there is no $b' \in B$ with $b' > b$). In other words, the maximal option of a factorization of $x \rightarrow y$ is of the following form.
    \[
    \xymatrix{ x \ar[r]\ar[dr] & z \ar[d]\ar[r] & y \\
    & s(g) & 
    }
    \]
    Note that as $x \in B$, $B$ is non-empty. 
    We show that the factorization
    \[
\xymatrix{ x \ar@/^1pc/[rrrr] \ar[rr] & & b \ar[rr] & &y
}    
\]

    exhibits $x\rightarrow y$ as an element of $\left< \AF \cup \Gamma_f \right>$.
    Since $x\rightarrow y$ is in $\F$ by assumption,  the following pullback square shows that $x\rightarrow b$ is in $\F$ too.
    \[
    \begin{tikzcd}
    x \arrow[r] \arrow[d, ""']  \arrow[dr, phantom, "\lrcorner", very near start] & x \arrow[d, ""] \\
    b \arrow[r, ""'] & y 
    \end{tikzcd}
    \]
We know that $x \rightarrow s(g)$ is a weak equivalence before localizing.
As weak equivalences are decomposable and $x \rightarrow s(g)$ factors over $b$, $x \rightarrow b$ is a weak equivalence and consequently, $x \rightarrow b \in \AF$. \\
Now consider the following commutative diagram
    \[\begin{tikzcd}
	b & {s(g)} & {s(g)} \\
	y & w & {t(g)}
	\arrow[from=1-1, to=1-2]
	\arrow[from=1-1, to=2-1]
	\arrow[from=1-2, to=1-3]
	\arrow[from=1-2, to=2-2]
	\arrow[from=1-3, to=2-3, "g"]
	\arrow[from=2-1, to=2-2]
	\arrow["\lrcorner"{anchor=center, pos=0.125, rotate=180},draw=none, from=2-2, to=1-1]
	\arrow[from=2-2, to=2-3]
    \end{tikzcd}\]
    where $w$ is defined by the left square being a pushout diagram. Note that the universal property of pushouts gives us a map $w \rightarrow t(g)$. By our construction of $b$, the left square is also a pullback diagram. 
        The same argument also makes the right hand square a pullback since $s(g) \leq w$. This implies that $b \rightarrow y$ is a pullback of a pullback of $g$. 
    Therefore we have that $x \rightarrow y$ can be written as a composite of an acyclic fibration and an arrow that is a pullback of an element of $\Gamma_f$, which implies that $$x \rightarrow y \in \left< AF \cup \Gamma_f \right>.$$
\end{proof}

Putting together Lemma \ref{lem:subset} and Lemma \ref{lem:supset} now finally gives us the desired description of right Bousfield localization.
\begin{theorem}\label{thm:golden}
    For any model structure $(\W,\AF)$ on a lattice and a short edge $f$, the closure with respect to the set of Golden Arrows $\Gamma_f$ corresponding to $f$ coincides with right localization at $f$, i.e.
    \[
    \Gamma_f(\W,\AF) = R_f(\W,\AF) .
    \]

More precisely, we have $R_f(\W,\AF)=(R_f(\W), \left<\AF \cup \Gamma_f\right>),$
where $R_f(\W)$ is given by Construction \ref{con:rf}, and $\left<\AF\cup \Gamma_f\right>$ is the smallest transfer system containing $\AF$ and $\Gamma_f$.
\qed
    \end{theorem}

We conclude this section by noting that for a fixed weak equivalence set $W$, the set $\Gamma_f$ is minimal among all sets $M$ such that $M(\W,\AF) = R_f(\W,\AF)$ (see Definition \ref{def:model}) in the following sense: there is no strict subset $M \subsetneq \Gamma_f$ such that $M(\W,\AF) = R_f(\W,\AF)$ for all possible model structures $(\W, \AF)$ with this same $\W$. This illustrated in Example \ref{ex:squareexample}, where we have a set of two Golden Arrows $\Gamma_f=\{g_1, g_2\}$, but adding just $g_1$ or just $g_2$ to the old acyclic fibrations would not result in the correct transfer system. 

\section{Left Localization and Co-Transfer Systems}\label{sec:dual_story}

As outlined in Section \ref{sec:bousfieldloc}, right localization increases $\W$ and $\AF$ while left localization increases $\W$ and $\AC$. 
In the previous section we described the change in $\AF$ after right localization in terms of a set of arrows we called Golden Arrows, associated to the morphism we wished to right localize at. 
In this section is to provide the analogous set of arrows, called \emph{Copper Arrows}, for left localization at a given morphism.  
We begin by defining the dual notion to a transfer system.
\begin{definition}\label{def:cotransfer}
	A \textit{co-transfer system} on a category $\Cate$ is a wide subcategory that is closed under pushouts along arbitrary morphisms of $\Cate$.
\end{definition}

Dually to Proposition \ref{prop:transferaf}, the acyclic cofibrations of a model structure form a co-transfer system, as acyclic cofibrations are closed under pushouts.
Similarly to our discussion after Definition \ref{def1.3}, the data of a model structure is there completely determined by giving its weak equivalences $\W$ and its co-transfer system of acyclic cofibrations $\AC$. We can therefore denote a model structure by $\llbracket \W,\AC\rrbracket_\Cate$, dropping the subscript when there is no risk of ambiguity.

\begin{definition}
    Let $\llbracket \W,\AC \rrbracket_\Cate$ be a model structure on a lattice $\Cate$ and $f$ a short arrow in $\Cate$.
    The \textit{Copper Arrows associated to $f$} by are defined by the set
    \[
    \kappa_f = \{   s(\sigma)\rightarrow t(\sigma)\ \vert \ \sigma \in L_f(\W)\setminus\W, \;  s(\sigma) \text{ is minimal in } [s(\sigma)]_{\W} \text{ and $t$ is minimal in } [t(\sigma)]_{\W}    \}.
    \]
\end{definition}

\begin{example}
Consider the model structure below on $[2] \times [1]$ below, where the shaded regions depict the weak equivalences, and the acylic cofibrations are the hooked arrows within those. \label{example:copper_arrows}
    \begin{figure}[H]
\begin{tikzpicture}[scale=1.1,
  hookarrow/.style={{Hooks[right]}->}]
\draw[red!30, fill, rounded corners] (0.6, 1.6) rectangle (3.4 ,2.4);
\draw[red!30, fill, rounded corners] (1.6, 2.4) rectangle (3.4 ,0.6);
\draw[red!30, fill, rounded corners] (0.6, 0.6) rectangle (1.4 ,1.4);

 \foreach \x in {1,2, 3}
 \foreach \y in {1,2}
    \fill (\x,\y) circle (1mm);
\node (A) at (1,2) {};
\node (B) at (2,2) {};
\node (C) at (3,2) {};
\node (D) at (1,1) {};
\node (E) at (2,1) {};
\node (F) at (3,1) {};
\node (G) at (2,3) {};
\node (H) at (1,3) {};

\draw[thick, hookarrow](E) edge (C);
\draw[thick, hookarrow](A) edge (B);
\draw[thick, hookarrow](E) edge (B);
\draw[thick, hookarrow](B) edge (C);
\draw[thick, hookarrow](F) edge (C);
\draw[thick, hookarrow](A) edge[bend left]  (C);
 \draw[line width=.5mm, black, -{stealth}] (D) edge node[above]{$f$} (E);
\end{tikzpicture}
\end{figure}
The set of Copper Arrows associated to the arrow labeled $f$ by the dashed arrows $c_1$ and $c_2$ below.
\label{example:copper_arrows}
    \begin{figure}[H]
\begin{tikzpicture}[scale=1.1,
  hookarrow/.style={{Hooks[right]}->}]
\draw[red!30, fill, rounded corners] (0.6, 1.6) rectangle (3.4 ,2.4);
\draw[red!30, fill, rounded corners] (1.6, 2.4) rectangle (3.4 ,0.6);
\draw[red!30, fill, rounded corners] (0.6, 0.6) rectangle (1.4 ,1.4);

 \foreach \x in {1,2, 3}
 \foreach \y in {1,2}
    \fill (\x,\y) circle (1mm);
\node (A) at (1,2) {};
\node (B) at (2,2) {};
\node (C) at (3,2) {};
\node (D) at (1,1) {};
\node (E) at (2,1) {};
\node (F) at (3,1) {};
\node (G) at (2,3) {};
\node (H) at (1,3) {};

\draw[thick, hookarrow](E) edge (C);
\draw[thick, hookarrow](A) edge (B);
\draw[thick, hookarrow](E) edge (B);
\draw[thick, hookarrow](B) edge (C);
\draw[thick, hookarrow](F) edge (C);
\draw[thick, hookarrow](A) edge[bend left]  (C);
    \draw[line width=2.5pt, gold, -{stealth}, dashed] (D) edge[bend left] node[left]{\color{black}$c_1$} (A);
    \draw[line width=2.5pt, gold, -{stealth}, dashed] (D) edge[bend right] node[below]{\color{black}$c_2$} (E);
      \draw[line width=.5mm, black, -{stealth}] (D) edge node[above]{$f$} (E);
\end{tikzpicture}
\end{figure}

\end{example}

With this in place, we now see that the acyclic cofibrations after left localization at a short arrow $f$ are given by the smallest co-transfer system containing the original $\AC$ as well as $\kappa_f$ in the following way. As the proof follows the same strategy as that of Theorem \ref{thm:golden}, we will prioritize brevity over detail. 

\begin{theorem}\label{thm:copper}
    For any model structure $\llbracket \W,\AC \rrbracket$ on a lattice, the acyclic cofibrations $L_f(\AC)$ of the left localization along a short arrow $f$ are given by the closure under composition of the union of the acyclic cofibrations $\AC$ and all pushouts of the copper arrows $\kappa_f$, i.e. by the smallest co-transfer system containing $\AC$ and $\kappa_f$. 
\end{theorem}


\begin{proof}
	First, observe that by the factorization axiom of model categories (MC5) all Copper Arrows must become acyclic cofibrations upon left localization, so $\kappa_f \subset L_f(\AC)$. 
    It follows that all their pushouts along any morphism and compositions with morphisms in $\AC$ must become acyclic cofibrations upon left localization.\\
    For the converse we show that all new acyclic cofibrations, i.e.~ elements of $L_f(\AC) \setminus \AC$, can be factored as a composition of a pushout of a copper arrow and an element of $\AC$. Consider a morphism 
    \[ \phi: x \rightarrow y \in L_f(\AC) \setminus \AC, \]
    then $x$ and $y$ must have been in different weak equivalence classes before localization. Let $a$ be a minimal element of $[x]_\W$ and $z$ be a minimal object in $[y]_\W$ of $y$ such that $\phi$ factors through $z$
	\begin{equation}\label{eq:AC_decomposition}
	    x \rightarrow z \rightarrow y .
	\end{equation}
    The map \(a\rightarrow z\) is a Copper Arrow by definition, and $x \rightarrow z$ is a pushout of it along the morphism $a \rightarrow x$.\\
 On the other hand, since $\phi \in L_f(\AC)$, it was already a cofibration before localizing, as cofibrations do not change under left localization. 
 Also, cofibrations are closed under pushouts, therefore the arrow $z\rightarrow y$ in the
 following pushout square is a cofibration.
    \[
    \begin{tikzcd}
    x \arrow[r] \arrow[d, "\phi"'] & z \arrow[d, ""] \\
    y \arrow[r, ""']   & y \arrow[ul, phantom, "\ulcorner", very near start]
    \end{tikzcd}
    \]
     But we have $[z]_{\W} = [y]_{\W}$ by construction, hence $z \rightarrow y$ was actually an acyclic cofibration before localization. 
    With this, observe that (\ref{eq:AC_decomposition}) exhibits $\phi$ as a composition of a pushout of a Copper Arrow and an acyclic cofibration.
\end{proof}

\section{Localization and Saturation}\label{sec:saturation}

We saw in Proposition \ref{prop:totalordergoldenarrow} that for a total order $[n]$, the Golden Arrow construction is particularly simple as only one arrow is needed.
Together with the classification of transfer systems on $[n]$ given in \cite[Theorem 20]{BBR}, this led to the result that on $[n]$, any model structure can be obtained from the trivial model structure (i.e. the unique model structure with only trivial weak equivalences) via a sequence of left and right Bousfield localizations \cite[Theorem 5.12]{BOOR23}. 

It turns our the same holds true for the subgroup lattice of $[1]\times[1]$, which can be verified directly from observing the effect of localization on the 23 possible model structures. 
We illustrate this in Figure \ref{fig:gridex}.
\begin{figure}
    \centering
    \includegraphics[scale = 1]{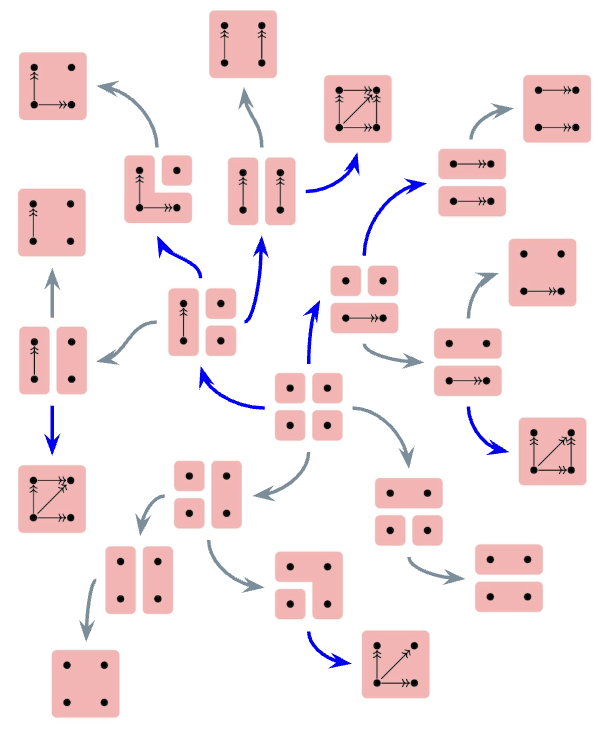}
    \caption{Localizations on the $[1] \times [1]$ grid. The blue arrows denote right localization, and the grey arrows a left localization.}
    \label{fig:gridex}
\end{figure}

However, in general one cannot obtain every model structure via localizations of the trivial one.
Consider the model structure on the subgroup lattice $[2]\times[1]$ depicted in Figure \ref{fig:anotherpic}.
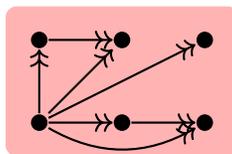
\begin{figure}[H]
\begin{tikzpicture}[scale=1.1]
\draw[red!30, fill, rounded corners] (0.6, 0.6) rectangle (3.4 ,2.4);
 \foreach \x in {1,2, 3}
 \foreach \y in {1,2}
    \fill (\x,\y) circle (1mm);
\node (A) at (1,2) {};
\node (B) at (2,2) {};
\node (C) at (3,2) {};
\node (D) at (1,1) {};
\node (E) at (2,1) {};
\node (F) at (3,1) {};

\draw[thick, ->>] (A) edge (B);
\draw[thick, ->>](D) edge (E);
\draw[thick, ->>] (D) edge (C);
\draw[thick, ->>](D) edge (A);
\draw[thick, ->>](E) edge (F);
\draw[thick, ->>](D) edge[bend right] (F);
\draw[thick, ->>](D) edge (B);
\end{tikzpicture}
\caption{A model structure where all morphisms are weak equivalences, and the acyclic fibrations are given by the double-pointed arrows.}
\label{fig:anotherpic}
\end{figure}
\vspace{-.50cm}
This is not the left Bousfield localization of any other model structure on $[2] \times [1]$ by the following reasoning. 
Left localization preserves acyclic fibrations, therefore any model structure that our example could be a left localization of would have to have the same acyclic fibrations and a strictly smaller weak equivalence set. 
However, the acyclic fibrations contain the long diagonal from (0,0) to (2,1). By decomposability of weak equivalences, any weak equivalence set containing the long diagonal has to contain all the morphisms in this category.
Therefore it is not possible to have a model structure with the same acyclic fibrations but fewer weak equivalences. 

The argument that the example is not a right Bousfield localization of any other model structure is relatively similar. 
The acyclic cofibrations in our example are the hooked arrows in the following picture.
\begin{figure}[H]
\begin{tikzpicture}[scale=1.1,
  hookarrow/.style={{Hooks[right]}->}]
\draw[red!30, fill, rounded corners] (0.6, 0.6) rectangle (3.4 ,2.4);
 \foreach \x in {1,2, 3}
 \foreach \y in {1,2}
    \fill (\x,\y) circle (1mm);
\node (A) at (1,2) {};
\node (B) at (2,2) {};
\node (C) at (3,2) {};
\node (D) at (1,1) {};
\node (E) at (2,1) {};
\node (F) at (3,1) {};
\node (G) at (2,3) {};
\node (H) at (1,3) {};

\draw[thick, hookarrow](E) edge (C);
\draw[thick, hookarrow](E) edge (B);
\draw[thick, hookarrow](B) edge (C);
\draw[thick, hookarrow](F) edge (C);
\draw[thick, hookarrow](A) edge[bend left]  (C);
\end{tikzpicture}
\end{figure}
As right Bousfield localization preserves acyclic cofibrations, any model structure that the original is a right localization of would need to have the same acyclic cofibrations.
By assumption, it must have strictly fewer weak equivalences than the original model structure above, i.e. it must be the right localization of a model structure with its weak equivalences being the shaded regions in the picture below.
\begin{figure}[H]
\begin{tikzpicture}[scale=1.1,
  hookarrow/.style={{Hooks[right]}->}]
\draw[red!30, fill, rounded corners] (0.6, 1.6) rectangle (3.4 ,2.4);
\draw[red!30, fill, rounded corners] (1.6, 2.4) rectangle (3.4 ,0.6);
\draw[red!30, fill, rounded corners] (0.6, 0.6) rectangle (1.4 ,1.4);

 \foreach \x in {1,2, 3}
 \foreach \y in {1,2}
    \fill (\x,\y) circle (1mm);
\node (A) at (1,2) {};
\node (B) at (2,2) {};
\node (C) at (3,2) {};
\node (D) at (1,1) {};
\node (E) at (2,1) {};
\node (F) at (3,1) {};
\node (G) at (2,3) {};
\node (H) at (1,3) {};

\draw[thick, hookarrow](E) edge (C);
\draw[thick, hookarrow](E) edge (B);
\draw[thick, hookarrow](B) edge (C);
\draw[thick, hookarrow](F) edge (C);
\draw[thick, hookarrow](A) edge[bend left]  (C);
\end{tikzpicture}
\end{figure}
Observe that the only model structures with this weak equivalence shape and these acyclic cofibrations are given by the two diagrams below.
We determined this from the list of 182 model structures on $[2] \times [1]$ generated by \cite{Bal25}, or do a quick calculation using Theorem \ref{thm:tmin}.
\begin{figure}[H]\label{fig:nolocalization}
\begin{tikzpicture}[scale=1.1]
\draw[red!30, fill, rounded corners] (0.6, 1.6) rectangle (3.4 ,2.4);
\draw[red!30, fill, rounded corners] (1.6, 2.4) rectangle (3.4 ,0.6);
\draw[red!30, fill, rounded corners] (0.6, 0.6) rectangle (1.4 ,1.4);

 \foreach \x in {1,2, 3}
 \foreach \y in {1,2}
    \fill (\x,\y) circle (1mm);
\node (A) at (1,2) {};
\node (B) at (2,2) {};
\node (C) at (3,2) {};
\node (D) at (1,1) {};
\node (E) at (2,1) {};
\node (F) at (3,1) {};
\node (G) at (2,3) {};
\node (H) at (1,3) {};
\end{tikzpicture} 
\hspace{3cm}
\begin{tikzpicture}[scale=1.1]
\draw[red!30, fill, rounded corners] (0.6, 1.6) rectangle (3.4 ,2.4);
\draw[red!30, fill, rounded corners] (1.6, 2.4) rectangle (3.4 ,0.6);
\draw[red!30, fill, rounded corners] (0.6, 0.6) rectangle (1.4 ,1.4);

 \foreach \x in {1,2, 3}
 \foreach \y in {1,2}
    \fill (\x,\y) circle (1mm);
\node (A) at (1,2) {};
\node (B) at (2,2) {};
\node (C) at (3,2) {};
\node (D) at (1,1) {};
\node (E) at (2,1) {};
\node (F) at (3,1) {};
\node (G) at (2,3) {};
\node (H) at (1,3) {};

\draw[thick, ->>](E) edge (F);

\end{tikzpicture}
\end{figure}
Right localizing either of these at a short arrow will not give rise to a model structure with the original $\AF$.
Consider instead right localizing at the arrow $(0,0) \rightarrow (1,0)$ or $(0,0)\rightarrow (0,1)$ in the left diagram. 
By our results from the previous section the resulting model structure would be the following.
\begin{figure}[H]
\begin{tikzpicture}[scale=1.1]
\draw[red!30, fill, rounded corners] (0.6, 0.6) rectangle (3.4 ,2.4);

 \foreach \x in {1,2, 3}
 \foreach \y in {1,2}
    \fill (\x,\y) circle (1mm);
\node (A) at (1,2) {};
\node (B) at (2,2) {};
\node (C) at (3,2) {};
\node (D) at (1,1) {};
\node (E) at (2,1) {};
\node (F) at (3,1) {};
\node (G) at (2,3) {};
\node (H) at (1,3) {};
\draw[thick, ->>](D) edge (A);
\draw[thick, ->>](D) edge (B);
\draw[thick, ->>](D) edge (E);
\draw[thick, ->>](D) edge (C);
\draw[thick, ->>](D) edge[bend right] (F);
\end{tikzpicture} 
\end{figure}
Similarly, if we right localize at the arrow $(0,0) \rightarrow (1,0)$ or $(0,0)\rightarrow (0,1)$ in the diagram on the right, by our results from the previous section, the resulting model structure would be the following.
\begin{figure}[H]
\begin{tikzpicture}[scale=1.1]
\draw[red!30, fill, rounded corners] (0.6, 0.6) rectangle (3.4 ,2.4);

 \foreach \x in {1,2, 3}
 \foreach \y in {1,2}
    \fill (\x,\y) circle (1mm);
\node (A) at (1,2) {};
\node (B) at (2,2) {};
\node (C) at (3,2) {};
\node (D) at (1,1) {};
\node (E) at (2,1) {};
\node (F) at (3,1) {};
\node (G) at (2,3) {};
\node (H) at (1,3) {};
\draw[thick, ->>](D) edge (A);
\draw[thick, ->>](D) edge (B);
\draw[thick, ->>](D) edge (E);
\draw[thick, ->>](D) edge (C);
\draw[thick, ->>](D) edge[bend right] (F);
\draw[thick, ->>](E) edge (F);
\end{tikzpicture} 
\end{figure}
The above example shows that, unlike in $[n]$ and $[1] \times [1]$, it is \emph{not} true in general that every model structure can be obtained as a sequence of left and right localizations from the trivial one.
It is nonetheless an intriguing question to consider which model structures can be obtained from which others via localization, and what properties localization may preserve.

\subsection{Saturation}
In this subsection we address how saturation of the transfer system of acyclic fibrations is affected by right localization. We start by recalling the definition of saturation.
\begin{definition}
    A transfer system $T$ is \emph{saturated} if it  satisfies the following property:
    \[\text{if }x \leq y \leq z \in \mathcal{C}\text{ and }x \rightarrow y, x \rightarrow z \in T\text{, then } y \rightarrow z \in T .\]
\end{definition}
Pictorially the saturation condition of a transfer system can be visualized as follows.
\begin{figure}[H]
\begin{tikzpicture}[scale=1.1]

 \foreach \x in {1,2, 3}
 \foreach \y in {1}
    \fill (\x,\y) circle (1mm);
\node (A) at (1,1) {};
\node (D) at (2,1) {};
\node (H) at (3,1) {};
\draw[thick, ->](A) edge (D);
\draw[thick, ->](A) edge[bend right]  (H);
\end{tikzpicture}
$\implies $
\begin{tikzpicture}[scale=1.1]

 \foreach \x in {1,2, 3}
 \foreach \y in {1}
    \fill (\x,\y) circle (1mm);
\node (A) at (1,1) {};
\node (D) at (2,1) {};
\node (H) at (3,1) {};
\draw[thick, ->](A) edge (D);
\draw[thick, ->](A) edge[bend right]  (H);
\draw[thick, ->](D) edge (H);
\end{tikzpicture}.
\end{figure}

Notice that on a transfer system being saturated is equivalent to satisfying the two-out-of-three condition.

\begin{remark}\label{rem:3outof4}
On the two-dimensional grid $[m] \times [n]$, further characterizations for saturation hold. 
Firstly, Proposition 2.10 of \cite{HMOO} states that for every saturated transfer system $T$, $T=\left<S\right>$ where $S$ is the set of all short arrows contained in $T$. 
Being able to work with short arrows instead of the whole transfer system has numerous practical advantages. 
For example, Theorem 3.2 of \cite{HMOO} declares that for a set of short arrows $S$ on the $[m] \times [n]$ grid, $S$ generates a saturated transfer system if and only if:
\begin{itemize}
\item the pullbacks of elements in $S$ are also in $S$,
\item if three out of the four sides of a square are in $S$, then so is the fourth.
\end{itemize}

\end{remark}

Our characterizations of localization applied to the lattice $[2]\times [1]$ show that the acyclic fibrations of model structures not arising as a sequence of left and right Bousfield localizations from the trivial model structure always form an unsaturated transfer system. 
See below for the full list of unsaturated transfer systems that do not arise as from a right or left localization.
We organized the list into groups according to their configuration of short arrows in the left square. \\

        \begin{figure}[H]
        \begin{tikzpicture}[scale=1.1]
\draw[red!30, fill, rounded corners] (0.6, 0.6) rectangle (3.4 ,2.4);
 \foreach \x in {1,2,3}
 \foreach \y in {1,2}
    \fill (\x,\y) circle (1mm);
\node (A) at (1,2) {};
\node (B) at (2,2) {};
\node (C) at (3,2) {};
\node (D) at (1,1) {};
\node (E) at (2,1) {};
\node (F) at (3,1) {};

\draw[thick, ->>](D) edge (E);
\draw[thick, ->>] (D) edge (C);
\draw[thick, ->>](D) edge (A);
\draw[thick, ->>](D) edge (B);
\draw[thick, ->>](D) edge[bend right] (F);
\draw[thick, ->>](E) edge[right] (B);
\end{tikzpicture}
\begin{tikzpicture}[scale=1.1]
\draw[red!30, fill, rounded corners] (0.6, 0.6) rectangle (3.4 ,2.4);
 \foreach \x in {1,2,3}
 \foreach \y in {1,2}
    \fill (\x,\y) circle (1mm);
\node (A) at (1,2) {};
\node (B) at (2,2) {};
\node (C) at (3,2) {};
\node (D) at (1,1) {};
\node (E) at (2,1) {};
\node (F) at (3,1) {};

\draw[thick, ->>](D) edge (E);
\draw[thick, ->>] (D) edge (C);
\draw[thick, ->>](D) edge (A);
\draw[thick, ->>](E) edge (F);
\draw[thick, ->>](D) edge (B);
\draw[thick, ->>](D) edge[bend right] (F);
\draw[thick, ->>](E) edge[right] (B);
\end{tikzpicture}
\begin{tikzpicture}[scale=1.1]
\draw[red!30, fill, rounded corners] (0.6, 0.6) rectangle (3.4 ,2.4);
 \foreach \x in {1,2,3}
 \foreach \y in {1,2}
    \fill (\x,\y) circle (1mm);
\node (A) at (1,2) {};
\node (B) at (2,2) {};
\node (C) at (3,2) {};
\node (D) at (1,1) {};
\node (E) at (2,1) {};
\node (F) at (3,1) {};

\draw[thick, ->>](B) edge[right] (C);
\draw[thick, ->>](D) edge (E);
\draw[thick, ->>] (D) edge (C);
\draw[thick, ->>](D) edge (A);
\draw[thick, ->>](E) edge (F);
\draw[thick, ->>](D) edge (B);
\draw[thick, ->>](D) edge[bend right] (F);
\draw[thick, ->>](E) edge[right] (B);
\draw[thick, ->>](E) edge[right] (C);
\end{tikzpicture}
\captionsetup{labelformat=empty}
\caption{Group 1}
\end{figure}
 \begin{figure}[H]
\begin{tikzpicture}[scale=1.1]
\draw[red!30, fill, rounded corners] (0.6, 0.6) rectangle (3.4 ,2.4);
 \foreach \x in {1,2,3}
 \foreach \y in {1,2}
 \fill (\x,\y) circle (1mm);
\node (A) at (1,2) {};
\node (B) at (2,2) {};
\node (C) at (3,2) {};
\node (D) at (1,1) {};
\node (E) at (2,1) {};
\node (F) at (3,1) {};
\draw[thick, ->>] (A) edge (B);
\draw[thick, ->>](D) edge (E);
\draw[thick, ->>](D) edge (A);
\draw[thick, ->>](D) edge (B);
\draw[thick, ->>](D) edge[bend right] (F);
\draw[thick, ->>](E) edge[right] (B);
\end{tikzpicture}
\begin{tikzpicture}[scale=1.1]
\draw[red!30, fill, rounded corners] (0.6, 0.6) rectangle (3.4 ,2.4);
 \foreach \x in {1,2,3}
 \foreach \y in {1,2}
    \fill (\x,\y) circle (1mm);
\node (A) at (1,2) {};
\node (B) at (2,2) {};
\node (C) at (3,2) {};
\node (D) at (1,1) {};
\node (E) at (2,1) {};
\node (F) at (3,1) {};
\draw[thick,   ->>] (A) edge (B);
\draw[thick,   ->>](D) edge (E);
\draw[thick,   ->>] (D) edge (C);
\draw[thick,  ->>](D) edge (A);
\draw[thick,   ->>](D) edge (B);
\draw[thick,   ->>](D) edge[bend right] (F);
\draw[thick,   ->>](E) edge[right] (B);
\end{tikzpicture}
\begin{tikzpicture}[scale=1.1]
\draw[red!30, fill, rounded corners] (0.6, 0.6) rectangle (3.4 ,2.4);
 \foreach \x in {1,2,3}
 \foreach \y in {1,2}
    \fill (\x,\y) circle (1mm);
\node (A) at (1,2) {};
\node (B) at (2,2) {};
\node (C) at (3,2) {};
\node (D) at (1,1) {};
\node (E) at (2,1) {};
\node (F) at (3,1) {};
\draw[thick,   ->>] (A) edge (B);
\draw[thick,   ->>](D) edge (E);
\draw[thick,   ->>] (D) edge (C);
\draw[thick,   ->>](D) edge (A);
\draw[thick,   ->>](D) edge (B);
\draw[thick,   ->>](D) edge[bend right] (F);
\draw[thick,   ->>](E) edge[right] (B);
\draw[thick,   ->>](F) edge[right] (C);
\end{tikzpicture}
\begin{tikzpicture}[scale=1.1]
\draw[red!30, fill, rounded corners] (0.6, 0.6) rectangle (3.4 ,2.4);
 \foreach \x in {1,2,3}
 \foreach \y in {1,2}
    \fill (\x,\y) circle (1mm);
\node (A) at (1,2) {};
\node (B) at (2,2) {};
\node (C) at (3,2) {};
\node (D) at (1,1) {};
\node (E) at (2,1) {};
\node (F) at (3,1) {};
\draw[thick,   ->>] (A) edge (B);
\draw[thick,   ->>](D) edge (E);
\draw[thick,   ->>] (D) edge (C);
\draw[thick,   ->>](D) edge (A);
\draw[thick,   ->>](E) edge (F);
\draw[thick,   ->>](D) edge (B);
\draw[thick,   ->>](D) edge[bend right] (F);
\draw[thick,   ->>](E) edge[right] (B);
\end{tikzpicture}
\begin{tikzpicture}[scale=1.1]
\draw[red!30, fill, rounded corners] (0.6, 0.6) rectangle (3.4 ,2.4);
 \foreach \x in {1,2,3}
 \foreach \y in {1,2}
    \fill (\x,\y) circle (1mm);
\node (A) at (1,2) {};
\node (B) at (2,2) {};
\node (C) at (3,2) {};
\node (D) at (1,1) {};
\node (E) at (2,1) {};
\node (F) at (3,1) {};

\draw[thick,   ->>] (A) edge (B);
\draw[thick,   ->>](D) edge (E);
\draw[thick,   ->>] (D) edge (C);
\draw[thick,   ->>](D) edge (A);
\draw[thick,   ->>](E) edge (F);
\draw[thick,   ->>](D) edge (B);
\draw[thick,   ->>](D) edge[bend right] (F);
\draw[thick,   ->>](E) edge[right] (B);
\draw[thick,   ->>](E) edge[right] (C);
\end{tikzpicture}
\captionsetup{labelformat=empty}
\caption{Group 2}
\end{figure}
 \begin{figure}[H]
\begin{tikzpicture}[scale=1.1]
\draw[red!30, fill, rounded corners] (0.6, 0.6) rectangle (3.4 ,2.4);
 \foreach \x in {1,2,3}
 \foreach \y in {1,2}
    \fill (\x,\y) circle (1mm);
\node (A) at (1,2) {};
\node (B) at (2,2) {};
\node (C) at (3,2) {};
\node (D) at (1,1) {};
\node (E) at (2,1) {};
\node (F) at (3,1) {};

\draw[thick,   ->>] (A) edge (B);
\draw[thick,   ->>](D) edge (E);
\draw[thick,   ->>](D) edge (A);
\draw[thick,   ->>](D) edge (B);
\draw[thick,   ->>](D) edge[bend right] (F);
\end{tikzpicture}
\begin{tikzpicture}[scale=1.1]
\draw[red!30, fill, rounded corners] (0.6, 0.6) rectangle (3.4 ,2.4);
 \foreach \x in {1,2,3}
 \foreach \y in {1,2}
    \fill (\x,\y) circle (1mm);
\node (A) at (1,2) {};
\node (B) at (2,2) {};
\node (C) at (3,2) {};
\node (D) at (1,1) {};
\node (E) at (2,1) {};
\node (F) at (3,1) {};

\draw[thick,   ->>] (A) edge (B);
\draw[thick,   ->>](D) edge (E);
\draw[thick,   ->>](D) edge (A);
\draw[thick,   ->>](E) edge (F);
\draw[thick,   ->>](D) edge (B);
\draw[thick,   ->>](D) edge[bend right] (F);
\end{tikzpicture}
\begin{tikzpicture}[scale=1.1]
\draw[red!30, fill, rounded corners] (0.6, 0.6) rectangle (3.4 ,2.4);
 \foreach \x in {1,2,3}
 \foreach \y in {1,2}
    \fill (\x,\y) circle (1mm);
\node (A) at (1,2) {};
\node (B) at (2,2) {};
\node (C) at (3,2) {};
\node (D) at (1,1) {};
\node (E) at (2,1) {};
\node (F) at (3,1) {};
\draw[thick,   ->>] (A) edge (B);
\draw[thick,   ->>](D) edge (E);
\draw[thick,   ->>] (D) edge (C);
\draw[thick,   ->>](D) edge (A);
\draw[thick,   ->>](D) edge (B);
\draw[thick,   ->>](D) edge[bend right] (F);
\end{tikzpicture}
\begin{tikzpicture}[scale=1.1]
\draw[red!30, fill, rounded corners] (0.6, 0.6) rectangle (3.4 ,2.4);
 \foreach \x in {1,2,3}
 \foreach \y in {1,2}
    \fill (\x,\y) circle (1mm);
\node (A) at (1,2) {};
\node (B) at (2,2) {};
\node (C) at (3,2) {};
\node (D) at (1,1) {};
\node (E) at (2,1) {};
\node (F) at (3,1) {};
\draw[thick,   ->>] (A) edge (B);
\draw[thick,   ->>](D) edge (E);
\draw[thick,   ->>] (D) edge (C);
\draw[thick,   ->>](D) edge (A);
\draw[thick,   ->>](E) edge (F);
\draw[thick,   ->>](D) edge (B);
\draw[thick,   ->>](D) edge[bend right] (F);
\end{tikzpicture}
\captionsetup{labelformat=empty}
\caption{Group 3}
\end{figure}
 \begin{figure}[H]
 \begin{tikzpicture}[scale=1.1]
\draw[red!30, fill, rounded corners] (0.6, 0.6) rectangle (3.4 ,2.4);
 \foreach \x in {1,2,3}
 \foreach \y in {1,2}
    \fill (\x,\y) circle (1mm);
\node (A) at (1,2) {};
\node (B) at (2,2) {};
\node (C) at (3,2) {};
\node (D) at (1,1) {};
\node (E) at (2,1) {};
\node (F) at (3,1) {};
\draw[thick,   ->>](D) edge (E);
\draw[thick,   ->>](D) edge (A);
\draw[thick,   ->>](D) edge (B);
\draw[thick,   ->>](D) edge[bend right] (F);
\end{tikzpicture}
 \begin{tikzpicture}[scale=1.1]
\draw[red!30, fill, rounded corners] (0.6, 0.6) rectangle (3.4 ,2.4);
 \foreach \x in {1,2,3}
 \foreach \y in {1,2}
    \fill (\x,\y) circle (1mm);
\node (A) at (1,2) {};
\node (B) at (2,2) {};
\node (C) at (3,2) {};
\node (D) at (1,1) {};
\node (E) at (2,1) {};
\node (F) at (3,1) {};
\draw[thick,   ->>](D) edge (E);
\draw[thick,   ->>](D) edge (A);
\draw[thick,   ->>](E) edge (F);
\draw[thick,   ->>](D) edge (B);
\draw[thick,   ->>](D) edge[bend right] (F);
\end{tikzpicture}
\captionsetup{labelformat=empty}
\caption{Group 4}
\end{figure}

 \begin{figure}[H]
  \begin{tikzpicture}[scale=1.1]
\draw[red!30, fill, rounded corners] (0.6, 1.6) rectangle (3.4 ,2.4);

\draw[red!30, fill, rounded corners] (0.6, 1.4) rectangle (3.4 ,0.6);
 \foreach \x in {1,2,3}
 \foreach \y in {1,2}
    \fill (\x,\y) circle (1mm);
\node (A) at (1,2) {};
\node (B) at (2,2) {};
\node (C) at (3,2) {};
\node (D) at (1,1) {};
\node (E) at (2,1) {};
\node (F) at (3,1) {};

\draw[thick,   ->>] (A) edge (B);
\draw[thick,   ->>](D) edge (E);
\draw[thick,   ->>](D) edge[bend right] (F);
\end{tikzpicture}
 \begin{tikzpicture}[scale=1.1]
\draw[red!30, fill, rounded corners] (0.6, 1.6) rectangle (3.4 ,2.4);

\draw[red!30, fill, rounded corners] (0.6, 1.4) rectangle (3.4 ,0.6);
 \foreach \x in {1,2,3}
 \foreach \y in {1,2}
    \fill (\x,\y) circle (1mm);
\node (A) at (1,2) {};
\node (B) at (2,2) {};
\node (C) at (3,2) {};
\node (D) at (1,1) {};
\node (E) at (2,1) {};
\node (F) at (3,1) {};

\draw[thick,   ->>] (A) edge[bend left] (C);
\draw[thick,   ->>] (A) edge (B);
\draw[thick,   ->>](D) edge (E);
\draw[thick,   ->>](E) edge (F);
\draw[thick,   ->>](D) edge[bend right] (F);
\end{tikzpicture}
\captionsetup{labelformat=empty}
\caption{Group 5}
\end{figure}
In contrast to this, a model structure with a saturated transfer system as $\AF$ may, after localization, result in an model structure with an unsaturated transfer system as its acyclic fibrations.
\begin{example}
Consider the model structure with saturated acyclic fibrations below.
    \begin{figure}[H]
\begin{tikzpicture}[scale=1.1]

\draw[red!30, fill, rounded corners] (0.6, 0.6) rectangle (3.4, 1.4); 
\draw[red!30, fill, rounded corners] (1.6, 1.6) rectangle (3.4, 2.4); 
\draw[red!30, fill, rounded corners]  (0.6, 0.6)rectangle (1.4,2.4); 
 \foreach \x in {1,2,3}
 \foreach \y in {1,2}
    \fill (\x,\y) circle (1mm);
\node (A) at (1,2) {};
\node (B) at (2,2) {};
\node (C) at (3,2) {};
\node (D) at (1,1) {};
\node (E) at (2,1) {};
\node (F) at (3,1) {};

\draw[thick, ->>](D) edge (E);
\draw[thick,  ->>](D) edge (A);
\draw[thick, ->>](E) edge (F);
\draw[thick,  ->>](D) edge[bend right] (F);

\end{tikzpicture}

\end{figure}
Right localizing at any arrow, short or otherwise, produces an unsaturated set of acyclic fibrations.
    \begin{figure}[H]
\begin{tikzpicture}[scale=1.1]
\draw[red!30, fill, rounded corners] (0.6, 0.6) rectangle (3.4 ,2.4);
 \foreach \x in {1,2, 3}
 \foreach \y in {1,2}
    \fill (\x,\y) circle (1mm);
\node (A) at (1,2) {};
\node (B) at (2,2) {};
\node (C) at (3,2) {};
\node (D) at (1,1) {};
\node (E) at (2,1) {};
\node (F) at (3,1) {};

\draw[thick, ->>] (A) edge[bend left] (C);
\draw[thick, ->>] (A) edge (B);
\draw[thick, ->>](D) edge (E);
\draw[thick, ->>] (D) edge (C);
\draw[thick, ->>](D) edge (A);
\draw[thick, ->>](E) edge (F);
\draw[thick, ->>](D) edge (B);
\draw[thick, ->>](D) edge[bend right] (F);
\draw[thick, ->>](E) edge[right] (B);
\draw[thick, ->>](F) edge[right] (C);
\draw[thick, ->>](E) edge[right] (C);
\end{tikzpicture}
\end{figure}
\end{example} 
It turns out that this is indeed a part of a general phenomenon.

 \begin{proposition}\label{prop:unsat}
 Let $(\W,\AF)$ be a model structure on a lattice $P$ where $\AF$ is an unsaturated transfer system. Then for any right localization $R_f(\W,\AF)$, the transfer system $R_f(\AF)$ is also unsaturated. 
 \end{proposition} 
    \begin{proof}
    Let $\AF$ be an unsaturated transfer system that is part of a model structure on a lattice $P$ with weak equivalences $\W$.
    This means that there is a triangle of morphisms in $\W$
\[
\xymatrix{
x \ar@{->>}[rr]^\sim \ar@{->>}[dr]_\sim &   & z \\
& y \ar[ur]_\sim & 
}
\]
where $y \rightarrow z$ is not a fibration, witnessing unsaturation.
As this map is already a weak equivalence before localization, it will be one after localization too. 
Furthermore, fibrations before and after right localization agree, so $y \rightarrow z$ will not be a fibration post right localization, whereas $x \rightarrow z$ and $x \rightarrow y$ will of course be in $R_f(\AF)$.
Therefore the same triangle of morphisms is a witness for unsaturation in $R_f(\AF)$.
 \end{proof}

The following observation gives another interesting connection between saturation and model structures.

\begin{lemma}\label{lem:wequalst}
Let $T$ be a saturated transfer system on a finite lattice $P$. Then $(\W,\AF)=(T,T)$ is a model structure on $P$.
\end{lemma}

\begin{proof}
First note that because $T$ is saturated, it is closed under two-out-of-three axiom. 
Recall that ${}^\boxslash T$ denotes the set of morphisms in $P$ that have the left lifting property with respect to $T$. Therefore, we need to verify that $(T, {}^\boxslash T)$ forms a model structure where the fibrations are all morphisms with the right lifting property with respect to ${}^\boxslash T \cap W$.
By \cite[Theorem 4.5]{MORSVZ2}, the above construction yields a model structure if and only if ${}^\boxslash T \cap W$ is closed under pushouts. However, in our case
$${}^\boxslash T \cap W = {}^\boxslash T \cap T$$ is trivial as we are working in a poset.
\end{proof}

Lemma \ref{lem:wequalst} plays an important role in the next theorem, where we prove that every model structure \emph{with saturated acyclic fibrations} can be obtained via localizing the trivial model structure.
To do this we require the following construction \cite[Proposition 4.4]{HMOO}, which describes how to obtain any saturated transfer system on $[m]\times [n+1]$ from a unique smaller one, and vice versa.

To be more precise, Theorem 4.1 of \cite{HMOO} says that the data of a transfer system on $[m] \times [n+1]$ is equivalent to the data of a transfer system on $A \times [n]$ for $A \subset [m]$ together with the information of how $A$ is included in $[m]$. We will give the procedure as well as an example below.

\begin{construction}\label{con:saturated}
Let $A \subset [m]$ be a subposet, and $\left<S\right>$ be a saturated transfer system on $A \times [n]$. This will play the role of the ``smaller'' transfer system in $[m]\times [n+1]$.

Note that every transfer system on $A \times [n]$ gives rise to a transfer system on the grid $[|A|]\times [n]$ by simply ignoring the columns not in $A$. 
However, while a transfer system on $[l] \times [n]$ with $l \leq m$ can be considered as a transfer system on $A \times [n]$ for some $A$ with $|A|=l$, there is not necessarily unique way of doing so. 
Therefore, the information of how $A$ is a subposet of $[m]$ is indeed vital for this bijection. 
As every saturated transfer system is uniquely determined by its generating short arrows, see Remark \ref{rem:3outof4}, we will describe the process only in terms of those. 

{\bf From smaller to bigger:}
We explain how to make a bigger saturated transfer system out of a smaller one.
Given a saturated transfer system $\left<S\right>$ on $A \times [n]$, we add the following short arrows as generators.
\begin{enumerate}
\item Include $S$ into the $m \times (n+1)$-grid in the following way.
\begin{enumerate}
\item The $y$-coordinates of $A \times [n]$ are using the canonical inclusion $[n] \rightarrow [n+1]$. 
\item If every short horizontal arrow in $A \times [n]$ is of the form $(i,j) \rightarrow (i+1,j)$, then we also use the canonical inclusion $A \rightarrow [m]$ on the $x$-coordinates. 
\item If there are short horizontal arrows $(i,j) \rightarrow (t,j)$ for some $i \neq t-1$ in $A \times [n]$, include those as $(t-1,j) \rightarrow (t,j) \in [m] \times [n+1]$. 
\end{enumerate}
\item If (1b), skip to Step (4).
\item If (1c) and there is at least one short arrow $(i,j) \rightarrow (t,j)$ in $A \times [n]$ for $i \neq t-1$, duplicate the column with $x$-coordinate $i$ in $[m] \times [n+1]$ to the columns with $x$-coordinates $i, \ldots, t-1$. Then, add the horizontal short arrows $(x,j) \rightarrow (x+1,j)$ for $x=i$ and $0 \leq j \leq n$.  
\item Add the horizontal short arrows $(i-1,j) \rightarrow (i,j)$ for all $i \notin A$, $i>\min([m]\backslash A)$ and all $j \leq n+1$. 
\item For $i=0,\ldots,k$, where $k+1=\min([m]\backslash A)$, add the top vertical edges $(i,n) \rightarrow (i, n+1)$.
\item If the three-out-of-four condition of a little square requires it, fill in the top horizontal edges $(i-1,n+1) \rightarrow (i,n+1)$ between the arrows added in the previous step.
\end{enumerate}

\begin{remark}\label{rem:iota}
We could have equivalently described Construction \ref{con:saturated} as follows.
We now use Step 1 to define an inclusion $\iota: A \rightarrow [m]$.
\begin{itemize}
\item If we are in the case (1b), then $\iota$ is the canonical inclusion.
\item If we are in the case (1c) and there is a short horizontal arrow $(i,j) \rightarrow (t,j)$ and $i \neq t-1$, then we set $\iota(i)=t-1$. If for $i$ there is no such short arrow, we set $\iota(i)=i$. 
\end{itemize}
Instead of using the canonical inclusion in (1a) and then the ``thickening'' described in (1c), we could have included the components of the smaller transfer system $S$ using $\iota$ instead and then thickened each component to the left. 
\end{remark} 
The next example illustrates this construction. 
\begin{example}
 Let $n=2$, $m=5$ and $A$ be the subset $\{0,1,3,4\}$ of $[m]$.
Suppose we have the saturated transfer system $S$ on $A \times [n]$ as follows.
 \begin{figure}[H]
 \begin{tikzpicture}[scale=1.1]
 \foreach \x in {1,2,3,4}
 \foreach \y in {1,2}
  \fill (\x,\y) circle (1mm);
  \node (A) at (1,1)[label=below:$0$]{};
    \node (B) at (1,2) {};
    \node (C) at (2,1) [label=below:$1$]{};
    \node (D) at (2,2) {}; 
    \node (E) at (3,1)[label=below:$3$]{}; 
    \node (F) at (3,2) {}; 
    \node (G) at (4,1)[label=below:$4$]{}; 
\draw[line width=.5mm, black,-{stealth}]  (A) edge[left] (B);
\draw[line width=.5mm, black, -{stealth}]  (C) edge[left] (D);
\draw[line width=.5mm, black, -{stealth}]  (C) edge[left] (E);
 \end{tikzpicture}
\end{figure}
Now given the steps of the construction, note that we are case 1c) so we perform the appropriate inclusion $\iota$ from $[m]\times[n+1]$ to where $0 \mapsto 0$, $1 \mapsto 2$, $3 \mapsto 3$ and $4 \mapsto 4$. 
The result is the following.
 \begin{figure}[H]
 \begin{tikzpicture}[scale=1.1]
 \foreach \x in {1,2,3,4,5,6}
 \foreach \y in {1,2,3}
  \fill (\x,\y) circle (1mm);
    \node (A) at (1,1)[label=below:$0$]{};
    \node (B) at (1,2) {};
    \node (S) at (2,1)[label=below:$1$]{};
    \node (C) at (3,1)[label=below:$2$]{};
    \node (D) at (3,2) {}; 
    \node (E) at (4,1)[label=below:$3$]{}; 
    \node (T) at (5,1)[label=below:$4$]{};
    \node (R) at (6,1)[label=below:$5$]{};
\draw[line width=.5mm, black,-{stealth}]  (A) edge[left] (B);
\draw[line width=.5mm, black, -{stealth}]  (C) edge[left] (D);
\draw[line width=.5mm, black, -{stealth}]  (C) edge[left] (E);
 \end{tikzpicture}
\end{figure}
Now we perform Step 3) to get the following image.
 \begin{figure}[H]
 \begin{tikzpicture}[scale=1.1]
 \foreach \x in {1,2,3,4,5,6}
 \foreach \y in {1,2,3}
  \fill (\x,\y) circle (1mm);
    \node (A) at (1,1) {};
    \node (B) at (1,2) {};
    \node (F) at (2,2) {};
    \node (G) at (2,1) {};
    \node (C) at (3,1) {};
    \node (D) at (3,2) {}; 
    \node (E) at (4,1) {}; 
\draw[line width=.5mm, black,-{stealth}]  (A) edge[left] (B);
\draw[line width=.5mm, black, -{stealth}]  (C) edge[left] (D);
\draw[line width=.5mm, black, -{stealth}]  (C) edge[left] (E);

  \draw[line width=.7mm, orange, -{stealth}] (G) edge[left] (F);
  \draw[line width=.7mm, red, -{stealth}] (G) edge[left] (C);
   \draw[line width=.7mm, red, -{stealth}] (F) edge[left] (D);
 \end{tikzpicture}
\end{figure}
Step 4) imposes that we include the arrows in pink between the two rightmost columns, and Step 5) imposed we include the arrows in purple between the two top rows.
 \begin{figure}[H]
 \begin{tikzpicture}[scale=1.1]
 \foreach \x in {1,2,3,4,5,6}
 \foreach \y in {1,2,3}
  \fill (\x,\y) circle (1mm);
    \node (A) at (1,1) {};
    \node (B) at (1,2) {};
    \node (X) at (1,3) {};
    \node (Y) at (2,3) {};
    \node (F) at (2,2) {};
    \node (G) at (2,1) {};
    \node (C) at (3,1) {};
    \node (D) at (3,2) {}; 
    \node (E) at (4,1) {}; 
     \node (Z) at (5,1) {}; 
     \node (J) at (4,2) {};
     \node (K) at (5,2) {};
      \node (O) at (4,3) {};
    \node (P) at (5,3) {};
    \node (H) at (6,1) {};
      \node (I) at (6,2) {};
    \node (J) at (6,3) {};
    
\draw[line width=.5mm, black,-{stealth}]  (A) edge[left] (B);
\draw[line width=.5mm, black, -{stealth}]  (C) edge[left] (D);
\draw[line width=.5mm, black, -{stealth}]  (C) edge[left] (E);

  \draw[line width=.7mm, orange, -{stealth}] (G) edge[left] (F);
  \draw[line width=.7mm, red, -{stealth}] (G) edge[left] (C);
   \draw[line width=.7mm, red, -{stealth}] (F) edge[left] (D);

 \draw[line width=.7mm, pink, -{stealth}] (Z) edge[left] (H);
  \draw[line width=.7mm, pink, -{stealth}] (K) edge[left] (I);
 \draw[line width=.7mm, pink, -{stealth}] (P) edge[left] (J);

   \draw[line width=.7mm, purple, -{stealth}] (B) edge[left] (X);
   \draw[line width=.7mm, purple, -{stealth}] (F) edge[left] (Y);
 \end{tikzpicture}
\end{figure}
Hence the resulting transfer system on $[m] \times [n+1]$ is the following.
 \begin{figure}[H]
 \begin{tikzpicture}[scale=1.1]
 \foreach \x in {1,2,3,4,5,6}
 \foreach \y in {1,2,3}
  \fill (\x,\y) circle (1mm);
    \node (A) at (1,1) {};
    \node (B) at (1,2) {};
    \node (X) at (1,3) {};
    \node (Y) at (2,3) {};
    \node (F) at (2,2) {};
    \node (G) at (2,1) {};
    \node (C) at (3,1) {};
    \node (D) at (3,2) {}; 
    \node (E) at (4,1) {}; 
     \node (Z) at (5,1) {}; 
     \node (J) at (4,2) {};
     \node (K) at (5,2) {};
      \node (O) at (4,3) {};
    \node (P) at (5,3) {};
       \node (H) at (6,1) {};
      \node (I) at (6,2) {};
    \node (J) at (6,3) {};
    
\draw[line width=.5mm, black,-{stealth}]  (A) edge[left] (B);
\draw[line width=.5mm, black, -{stealth}]  (C) edge[left] (D);
\draw[line width=.5mm, black, -{stealth}]  (C) edge[left] (E);

  \draw[line width=.7mm, black, -{stealth}] (G) edge[left] (F);
  \draw[line width=.7mm, black, -{stealth}] (G) edge[left] (C);
   \draw[line width=.7mm, black, -{stealth}] (F) edge[left] (D);

 \draw[line width=.7mm, black, -{stealth}] (Z) edge[left] (H);
  \draw[line width=.7mm, black, -{stealth}] (K) edge[left] (I);
 \draw[line width=.7mm, black, -{stealth}] (P) edge[left] (J);

   \draw[line width=.7mm, black, -{stealth}] (B) edge[left] (X);
   \draw[line width=.7mm, black, -{stealth}] (F) edge[left] (Y);
 \end{tikzpicture}
\end{figure}
Notice that we did not perform Step 6) as it was not applicable in this setting. 
\end{example}

{\bf From bigger to smaller:}
As this construction forms part of a bijection, we have an inverse method for constructing a smaller transfer system from a larger one. 
So given a saturated transfer system $T$ on $[m] \times [n+1]$, we obtain a subset $A$ of $[m]$ and a transfer system $\left<S\right>$ on $A \times [n]$ as follows.

Consider the transfer system from the previous example. We will unravel the inverse construction. 
\begin{example} Consider the transfer system below on the lattice $[5] \times [2]$.

 \begin{figure}[H]
 \begin{tikzpicture}[scale=1.1]
 \foreach \x in {1,2,3,4,5,6}
 \foreach \y in {1,2,3}
  \fill (\x,\y) circle (1mm);
    \node (A) at (1,1)[label=below:$0$]{};
    \node (B) at (1,2) {};
    \node (X) at (1,3) {};
    \node (Y) at (2,3) {};
    \node (F) at (2,2) {};
    \node (G) at (2,1)[label=below:$1$]{};
    \node (C) at (3,1)[label=below:$2$]{};
    \node (D) at (3,2) {}; 
    \node (E) at (4,1)[label=below:$3$]{}; 
     \node (Z) at (5,1)[label=below:$4$]{}; 
     \node (J) at (4,2) {};
     \node (K) at (5,2) {};
      \node (O) at (4,3) {};
    \node (P) at (5,3) {};
       \node (H) at (6,1)[label=below:$5$]{};
      \node (I) at (6,2) {};
    \node (J) at (6,3) {};
    
\draw[line width=.5mm, black,-{stealth}]  (A) edge[left] (B);
\draw[line width=.5mm, black, -{stealth}]  (C) edge[left] (D);
\draw[line width=.5mm, black, -{stealth}]  (C) edge[left] (E);

  \draw[line width=.7mm, black, -{stealth}] (G) edge[left] (F);
  \draw[line width=.7mm, black, -{stealth}] (G) edge[left] (C);
   \draw[line width=.7mm, black, -{stealth}] (F) edge[left] (D);

 \draw[line width=.7mm, black, -{stealth}] (Z) edge[left] (H);
  \draw[line width=.7mm, black, -{stealth}] (K) edge[left] (I);
 \draw[line width=.7mm, black, -{stealth}] (P) edge[left] (J);

   \draw[line width=.7mm, black, -{stealth}] (B) edge[left] (X);
   \draw[line width=.7mm, black, -{stealth}] (F) edge[left] (Y);
 \end{tikzpicture}
\end{figure}

Our goal is to construct a transfer system on $A\times[1]$ for some $A \subseteq [5]$, so our first step is to remove the top row $[5]\times \{2\}$ of the grid and any arrows mapping to the nodes. 
The result is the next image.
\begin{figure}[H]
 \begin{tikzpicture}[scale=1.1]
 \foreach \x in {1,2,3,4,5,6}
 \foreach \y in {1,2}
  \fill (\x,\y) circle (1mm);
    \node (A) at (1,1) {};
    \node (B) at (1,2) {};
    \node (X) at (1,3) {};
    \node (Y) at (2,3) {};
    \node (F) at (2,2) {};
    \node (G) at (2,1) {};
    \node (C) at (3,1) {};
    \node (D) at (3,2) {}; 
    \node (E) at (4,1) {}; 
     \node (Z) at (5,1) {}; 
     \node (J) at (4,2) {};
     \node (K) at (5,2) {};
      \node (O) at (4,3) {};
    \node (P) at (5,3) {};
    \node (I) at (6,1) {};
    \node (J) at (6,2) {};
\draw[line width=.5mm, black,-{stealth}]  (A) edge[left] (B);
\draw[line width=.5mm, black, -{stealth}]  (C) edge[left] (D);
\draw[line width=.5mm, black, -{stealth}]  (C) edge[left] (E);
  \draw[line width=.7mm, black, -{stealth}] (G) edge[left] (F);
  \draw[line width=.7mm, black, -{stealth}] (G) edge[left] (C);
   \draw[line width=.7mm, black, -{stealth}] (F) edge[left] (D);
 \draw[line width=.7mm, black, -{stealth}] (Z) edge[left] (I);
  \draw[line width=.7mm, black, -{stealth}] (K) edge[left] (J);
 \end{tikzpicture}
\end{figure}
Then we collapse neighbouring duplicate columns to the right along parallel horizontal arrows, which gives rise to the following diagram.
 \begin{figure}[H]
 \begin{tikzpicture}[scale=1.1]
 \foreach \x in {1,2,3,4}
 \foreach \y in {1,2}
  \fill (\x,\y) circle (1mm);

    \node (A) at (1,1)[label=below:$0$]{};
    \node (B) at (1,2) {};
    \node (C) at (2,1)[label=below:$1$]{};
    \node (D) at (2,2) {}; 
    \node (E) at (3,1)[label=below:$3$]{}; 
    \node (F) at (3,2) {}; 
    \node (G) at (4,1)[label=below:$4$]{}; 
\draw[line width=.5mm, black,-{stealth}]  (A) edge[left] (B);
\draw[line width=.5mm, black, -{stealth}]  (C) edge[left] (D);
\draw[line width=.5mm, black, -{stealth}]  (C) edge[left] (E);
 \end{tikzpicture}
\end{figure}
We retain the rightmost labeling from each collapsed block of columns. 
So the labels in this picture are $A=\{0,1,3,4\}$. 
\end{example}
The synopsis of the example is the following. 
Given $T$, we remove the top row $[m] \times \{n+1\}$, as well as any arrows whose target is in the top row.
Then, if there are any neighboring columns that are duplicates of each other, we collapse them onto the rightmost of those columns. We give the newly collapsed column the $x$-coordinate of the leftmost column from before collapsing. In other words, the leftmost coordinate is the one to be taken into $A$. We are then left with a saturated transfer system on $A \times [n]$.

\end{construction}

We can finally prove our structural result on how model structures with saturated acyclic fibrations are linked via localizations.

\begin{theorem}\label{thm:saturatedzigzag}
Let $(\W,\AF)$ be a model structure on $[m] \times [n+1]$ where $\AF$ is a saturated transfer system. Then this model structure can be achieved by applying a sequence of left and right localizations to the trivial model structure.

\end{theorem}
\begin{proof}
We prove the claim in the following steps.
\begin{enumerate}
\item Create a model structure on $[m] \times [n+1]$ with our desired acyclic fibrations $\AF$, and weak equivalences contained in $\W$ but not necessarily equal to $\W$. This is achieved in the following substeps:
\begin{enumerate}
\item Create a model structure $(T_A,T_A)$ on a smaller grid $A \times [n]$, $A \subset [m]$, where $T_A$ is the transfer system corresponding to $\AF$ in Construction \ref{con:saturated}.
\item Use $T_A$ to build a model structure $(\AF_\iota,\AF_\iota)$, where $\AF_\iota$ is exactly the transfer system resulting from applying Step (1) of Construction \ref{con:saturated} to $T_A$.
\item Right localize $(\AF_\iota,\AF_\iota)$ further, according to the remaining steps of Construction \ref{con:saturated}, to obtain the model structure $(\AF,\AF)$. 
\end{enumerate}
\item Apply left localization to the previous step, which enlarges the weak equivalences to $\W$ but leaves $\AF$ intact.
\end{enumerate}

So let us begin with creating a model structure with the desired acyclic fibrations and argue inductively using Construction \ref{con:saturated}. There is a unique subposet $A \subset [m]$ and saturated transfer system $T_A$ such that applying the smaller-to-bigger part of Construction \ref{con:saturated} creates the transfer system $\AF$. 

By Lemma \ref{lem:wequalst}, $(T_A,T_A)$ is a model structure on $A \times [n]$. By our induction hypothesis, $(T_A, T_A)$ can be obtained from the trivial model structure by left and right localization, i.e. there are short arrows $g_i \in A \times [n]$ such that
\[
(T_A, T_A) = (R_{g_t} \circ L_{g_{t-1}} \circ \cdots \circ R_{g_2} \circ L_{g_1})(Id, Id).
\]
Note that for this presentation, we also allow some of the $g_i$ to be the identity. But this is not the only simplification. Because $(T_A, T_A)$ is a model structure with its weak equivalences equal to its acyclic fibrations, its only acyclic cofibrations are identity morphisms. By definition, left localization increases the set of acyclic cofibrations, which can only mean that there are no nontrivial left localizations in the above presentation. This means that we can in fact write
\[
(T_A, T_A) = (R_{g_k} \circ R_{g_k-1} \circ \cdots \circ R_{g_2} \circ R_{g_1})(Id, Id).
\]

Let $\iota: A \times [n+1] \longrightarrow [m]\times[n+1]$ be the inclusion constructed in Remark \ref{rem:iota}.
We can now examine what the composition of localizations 
\[
R_{\iota(g_k)} \circ R_{\iota(g_{k-1})} \circ \cdots \circ R_{\iota(g_2)} \circ R_{\iota(g_1)}
\]
does to the trivial model structure on $[m] \times [n+1]$. We call the resulting model structure $(\W_\iota, \AF_\iota)$.

In order to proceed with our proof, we make the following observations about$(\W_\iota, \AF_\iota)$.

\begin{itemize}
\item The arrows $\iota(g_i)$ are still short arrows, and they lie entirely in the blocks of the grid given by $\iota(A) \times [n]$.
\item By definition, $\iota(g_i) \in \AF_\iota$ for all $i$. 
\item By the model category axioms, every short arrow in $\W_\iota$ has to be either an acyclic cofibration or an acyclic fibration.
\end{itemize}
Since $\W_\iota$ was given by a sequence of right localizations of the trivial model structure, this means that no acyclic cofibrations were added, and therefore all short arrows in $\W_\iota$ are acyclic fibrations. 
This implies that $\W_\iota=\AF_\iota$.

Next, we look at $\W_\iota$ and where it lies in the grid. For this, we use the characterization of the effect of localization on weak equivalences, Theorem \ref{thm:legalw} and Construction \ref{con:rf}. 

Recall that Construction \ref{con:rf} gave us a weak equivalence set $R_f\W$ by first adding the arrow $f$ to a weak equivalence set $\W$, and then repeatedly closing under transfer system operations and under the two-out-of-three property of $\AF_{(n)} \circ \AC$. As in our special case we have a right localization of the trivial model structure, we have no nontrivial $\AC$ in our original model structure, and the process terminates at the second step each time. Therefore, 
\[
\W_\iota=\AF_\iota=\left< \iota(T_A)\right>.
\]
This means that $\W_\iota=\AF_\iota$ is simply including $T_A$ into the larger grid using $\iota$ and then closing under transfer system operations, which in fact only adds vertical pullbacks. In particular, $\W_\iota=\AF_\iota$ lies entirely in the blocks of the grid given by $\iota(A) \times [n]$ as well as in the vertical columns given by  $\{i\} \times [n], \,\, i \notin \iota(A)$, and $\AF_\iota$ is precisely the transfer system one ends up with after applying Step (1) of Construction \ref{con:saturated} to $T_A$.

Hence the arrows from steps $(3)$, $(4)$ and $(5)$ of Construction \ref{con:saturated}  are \emph{not} in $\W_\iota$.
Therefore, we can add those to 
$\AF_\iota$ via right localization.
The order in which we do this matters, as we cannot right localize at anything that is already a weak equivalence, so we first localize at the set from $(5)$, then $(4)$, then $(3)$.

We call the resulting model structure $(\W', \AF')$. By construction, we have added precisely those arrows in $\AF \setminus \AF_\iota$ to $\AF_\iota$ in those last right localizations, meaning that we now have $\AF'=\AF$. Again, as we have not added any acyclic cofibrations, we also have $\W'=\AF'$.

Thus, we have created a model structure $(\W', \AF)$ which arises as a sequence of right Bousfield localizations of the trivial model structure on $[m] \times [n+1]$, which has the correct acyclic fibrations and which satisfies $\W' \subseteq \W$. 

In order to finally enlarge $\W'$ to the desired $\W$ while leaving $\AF$ intact, we perform left localizations at all short arrows $s \in \W \setminus W'$. This is equivalent to performing a left localization of $(\W', \AF)$ at the acyclic cofibrations of $(\W,\AF)$.

Therefore, we have written $(\W,\AF)$ as a sequence of left and right localizations starting from the trivial model structure.
\end{proof}

\begin{remark}
In \cite{KEOOSSYYZ}, Ormsby, Osorno and coauthors 
provide an explicit description of saturated transfer systems on $P \times [n]$ for any finite lattice $P$ which indicates that the recursive construction on $[m] \times [n]$ might be generalized.
Therefore, we also expect a more general version of our saturation theorem to hold. 
\end{remark}

We can summarize our findings as follows.
\begin{itemize}
\item The left or right localization of a model structure with unsaturated acyclic fibrations will also have unsaturated acyclic fibrations.
\item If the acyclic fibrations are saturated, they might become unsaturated after certain right localizations.
\item On $[m] \times [n]$, if the acyclic fibrations are saturated, the model structure can be obtained as a sequence of left and right localizations from the trivial model.
\end{itemize}

\bigskip

\begin{center}
\includegraphics[width=8cm]{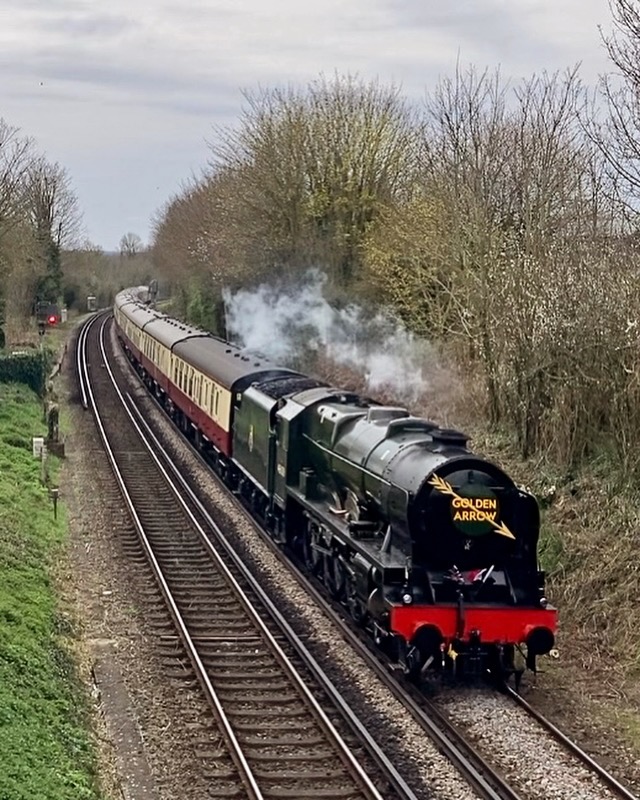}
\end{center}


\newcommand{\etalchar}[1]{$^{#1}$}

\end{document}